\documentclass[11pt,reqno]{amsart}

\usepackage[T1]{fontenc}
\usepackage[utf8]{inputenc}
\usepackage{lmodern}
\usepackage{microtype}
\usepackage[margin=1.02in]{geometry}
\usepackage{mathtools,amssymb,amsfonts}
\usepackage{enumitem}
\usepackage{xcolor}
\usepackage{booktabs,tabularx,array}
\usepackage{placeins}
\usepackage{hyperref}

\definecolor{deepblue}{HTML}{17365D}
\definecolor{burgundy}{HTML}{7B2131}
\hypersetup{
  colorlinks=true,
  linkcolor=deepblue,
  citecolor=burgundy,
  urlcolor=deepblue,
  pdftitle={Rainbow spanning configurations in optimally uniformly coloured percolated pseudorandom graphs}
}

\newtheorem{theorem}{Theorem}[section]
\newtheorem{lemma}[theorem]{Lemma}
\newtheorem{corollary}[theorem]{Corollary}

\theoremstyle{plain}
\newtheorem{definition}[theorem]{Definition}
\newtheorem{observation}[theorem]{Observation}
\theoremstyle{remark}
\newtheorem{remark}{Remark}[section]

\newcommand{\PP}{\mathbb{P}}

\newcommand{\cA}{\mathcal A}
\newcommand{\cF}{\mathcal F}
\newcommand{\cH}{\mathcal H}
\newcommand{\cT}{\mathcal T}

\newcommand{\ee}{\vec e}
\newcommand{\evec}{\vec e}

\newcommand{\cK}{\mathcal K}

\setlist[enumerate]{leftmargin=2.2em,itemsep=0.4em,topsep=0.4em}
\allowdisplaybreaks

\title[Exact-palette rainbow embeddings]
{Exact-palette rainbow embeddings in \\ uniformly coloured pseudorandom graphs}
\author{Elad Aigner-Horev, Dan Hefetz, Yury Person, Michael Trushkin}
\date{}

\begin{document}

\begin{abstract}
We study rainbow spanning configurations in bijumbled graphs whose edges are
coloured independently and uniformly from a prescribed palette.
For $n$-vertex $(p,\beta)$-bijumbled graphs with minimum degree at least
a fixed positive multiple of $pn$, we obtain rainbow perfect matchings
and Hamilton cycles with a sufficient palette surplus of order
$(\log n)/p$, assuming $pn=\omega(\log n)$ and $\beta\le cpn$ for a
sufficiently small constant $c$. For each prescribed spanning tree
of fixed maximum degree $\Delta\ge2$, a surplus of order
$L_{n,\Delta}(\log n)/p$ suffices under
$\beta\le cpn/L_{n,\Delta}$, where
$L_{n,\Delta}=\Delta^{5\sqrt{\log n}}$.
The palette surplus is sublinear under these hypotheses. These results
use a McDiarmid-type coupling and retain the discrepancy scales of
the relevant deterministic embedding theorems.

We also prove exact-palette results, using precisely as many colours as the
number of edges of the target configuration. After independent edge percolation at rate $\rho$, it is shown that a rainbow perfect matching or  Hamilton cycle exists asymptotically
almost surely when $\rho pn \ge C(\log n)^2$ and $\beta\le\gamma pn$ for appropriate constants $C$ and $\gamma$.
We obtain corresponding results for each prescribed bounded-degree
spanning tree and for clique factors under appropriate stronger
hypotheses. The exact-palette proofs construct spread measures from
uncoloured containment estimates and apply the rainbow threshold
theorem of Han and Yuan.



\end{abstract}

\maketitle

\section{Introduction}
\label{sec:introduction}

\noindent
\textbf{Synopsis.} A spanning graph configuration with $h$ edges trivially requires at least $h$ colours to be
rainbow under any edge-colouring. We investigate how closely the palette size can approach this elementary lower bound in uniformly coloured sparse pseudorandom graphs. For $(p,\beta)$-bijumbled graphs $G$, satisfying $\delta(G) = \Omega(pn)$, $pn = \omega(\log n)$ and $\beta = O(pn)$, our results show that it suffices to have a surplus of order $\log n/p$ for perfect matchings and Hamilton
cycles; for prescribed spanning trees having a fixed maximum degree $\Delta \geq 2$, a surplus of order $L_{n,\Delta} \log n/p$ suffices, where $L_{n,\Delta} := \Delta^{5\sqrt{\log n}}$, provided that $\beta = O(pn/L_{n,\Delta})$. Throughout these results, the bounds imposed on $p$ and $\beta$ are comparable to the best known bounds for the emergence of these configurations in uncoloured bijumbled graphs. 


A second strand of results we obtain shows that raising the edge density from $\Omega(\log n/n)$ to $\Omega((\log n)^2/n)$, whilst keeping the $\beta = O(p n)$ assumption, is enough to support rainbow embeddings of perfect matchings and Hamilton cycles whilst eliminating the palette surplus completely, leading to {\sl exact-palette} rainbow embeddings of said configurations. In fact, such embeddings are shown to emerge asymptotically almost surely (a.a.s. hereafter) in {\sl percolations} of the bijumbled host in which each edge of the latter is retained independently with probability $\rho := \rho(n) \in (0,1]$; provided that $\delta(G) = \Omega(pn)$, $\rho pn\geq C(\log n)^2$, and $\beta = O(pn)$. We also obtain exact-palette rainbow embedding results for prescribed bounded-degree spanning trees and clique factors in percolated bijumbled graphs. Here too, the asymptotic gap between our required parameter bounds and the best known bounds for the emergence of these configurations in the colourless setting is fairly mild. 

Our results pertaining to the survival of rainbow spanning configurations in bijumbled graphs and percolations thereof assert that despite being subjected to an exact-palette uniform colouring and the requirement that the configuration is rainbow, the structures in question exhibit a measure of {\sl statistical robustness}. This type of robustness was studied for $(n,d,\lambda)$-graphs by Chen--Chen--Han--Zhao~\cite{ChenChenHanZhao2025} in the colourless setting. 

\medskip

\noindent
\textbf{Edge colourings.} A subgraph is said to be \emph{rainbow} with respect to a given edge colouring of its host if all its edges receive distinct colours. By a \emph{$Q$-uniform colouring} of a graph, we mean a random colouring in which each edge is independently assigned a colour sampled uniformly at random from a prescribed palette of size $Q$; a \emph{$[Q]$-uniform colouring} uses the palette
$[Q]:=\{1,\ldots,Q\}$. For an $h$-edge target configuration, $Q-h$ is the
\emph{palette surplus}; the palette is called \emph{exact} if $Q=h$. A standard coupon-collector argument illustrates the substantial distinction
between exact-palette rainbow embeddings and those using a palette of size $h+o(h)$. 



\medskip
\noindent\textbf{Bijumbled graphs.} For a graph $G$ and sets $X,Y \subseteq V(G)$, define the ordered edge count
$$
 \evec_G(X,Y) := \bigl|\{(x,y) \in X \times Y : xy \in E(G)\}\bigr|;
$$
note that edges spanned by $X\cap Y$ are counted twice.


The following notion introduced in~\cite{KoRoScSiSk2007} (see also~\cite{Tho1987}) quantifies pseudorandomness of a graph. Given $0<p\leq1$
and $\beta\geq0$, the graph $G$ is said to be \emph{$(p,\beta)$-bijumbled} if
$$
 \left|\ee_G(X,Y) - p|X||Y|\right| \leq \beta \sqrt{|X||Y|}
$$
for all $X,Y\subseteq V(G)$. We often refer to the above inequality as the discrepancy (condition) and to $\beta$ as the discrepancy parameter. Note that $e_G(X,Y) = \ee_G(X,Y)$ holds whenever $X$ and $Y$ are disjoint; hence, we use both notations interchangeably in this case. 
A bipartite graph $H=(U,W;E)$ is
\emph{$(p,\beta)$-bijumbled} if
$$
 \left|e_H(X,Y)-p|X||Y|\right|
 \leq\beta\sqrt{|X||Y|}
$$
for all $X \subseteq U$ and $Y \subseteq W$.  

An \emph{$(n,d,\lambda)$-graph} is an $n$-vertex $d$-regular
graph whose second largest eigenvalue in absolute value is $\lambda$. The expander mixing lemma~\cite{AC1988} implies that every
$(n,d,\lambda)$-graph is $(d/n,\lambda)$-bijumbled; our results therefore
apply to such graphs by substituting $p=d/n$ and $\beta=\lambda$.


\medskip
\noindent\textbf{Uniformly coloured random graphs.}
Write $\mathbb G(n,p)$ for the binomial random graph. Table~\ref{tab:random-graph-literature} records representative results
for the configurations considered here and uniform colourings of the host. Here, and throughout the paper, $\log$ denotes the natural logarithm, and floor and ceiling signs are suppressed whenever they are not crucial for understanding.

\begin{table}[htbp]
\centering
\caption{{\small \textbf{Rainbow spanning configurations in uniformly coloured
random graphs.} Each row asserts emergence asymptotically almost surely.
Constants $C$ are sufficiently large and may depend on the fixed
parameters in that row. $\cT(n,\Delta)$ denotes the family of all $n$-vertex trees of maximum degree at most $\Delta$. Palette sizes are rounded to integers.}}
\label{tab:random-graph-literature}
\small
\renewcommand{\arraystretch}{1.18}
\setlength{\tabcolsep}{3.5pt}
\begin{tabular}{@{}>{\raggedright\arraybackslash}p{0.235\linewidth}>{\raggedright\arraybackslash}p{0.15\linewidth}>{\raggedright\arraybackslash}p{0.235\linewidth}>{\raggedright\arraybackslash}p{0.31\linewidth}@{}}
\hline
\textbf{Authors and reference} & \textbf{Palette size} & \textbf{Configuration} & \textbf{Parameters} \\[4pt]
\hline
Frieze--McKay~\cite{FriezeMcKay1994}
& $n-1$ & A spanning tree
& $p\geq(2+\varepsilon)\log n/n$ \\[4pt]
Frieze--Loh~\cite{FriezeLoh2014}
& $(1+\xi)n$ & Hamilton cycle
& $p=(1+\varepsilon)\log n/n$;
  $\varepsilon,\xi>125/\sqrt{\log\log n}$ \\[4pt]
Ferber--Krivelevich~\cite{FerberKrivelevich2016}
& $(1+\varepsilon)n$ & Hamilton cycle
& $p=(\log n+\log\log n+\omega(1))/n$;
  fixed $\varepsilon>0$ \\[4pt]
Bal--Frieze~\cite{BalFrieze2016}
& $n/2$ & Perfect matching
& $p\geq C\log n/n$; $2\mid n$ \\[4pt]
Bal--Frieze; Ferber~\cite{BalFrieze2016,Ferber2015Gaps}
& $n$ & Hamilton cycle
& $p\geq C\log n/n$ \\[4pt]
Bell--Frieze--Marbach~\cite{BellFriezeMarbach2024}
& $n-1$ & Prescribed $T\in\cT(n,\Delta)$
& $p\geq C(\Delta)\log n/n$; fixed $\Delta$ \\[4pt]
Ferber--Krivelevich; Johansson--Kahn--Vu~\cite{FerberKrivelevich2016,JohanssonKahnVu2008}
& $(1+\varepsilon)(k-1)n/2$ & $K_k$-factor
& $p\geq C(k,\varepsilon)n^{-2/k}(\log n)^{1/\binom{k}{2}}$;
  $k\mid n$ \\[4pt]
\hline
\end{tabular}
\end{table}

\FloatBarrier
\medskip
\noindent\textbf{Uniformly coloured randomly perturbed graphs.}
In a specific instantiation of the randomly perturbed model, a deterministic {\sl seed} $n$-vertex graph $G_0$, satisfying $\delta(G_0)\geq\delta n$, with $\delta>0$ independent of $n$, is fixed 
and then receives additional edges sampled through the binomial random
graph $\mathbb G(n,p)$; edges of the resulting union are then coloured uniformly from a prescribed palette. Table~\ref{tab:perturbed-rainbow-literature} records the progression from a large linear palette to exact palettes in this instantiation of the perturbed model; all results are with respect to uniform colourings. Throughout the table,
$\delta\in(0,1/2)$ and $\Delta\geq2$ are fixed. 

\begin{table}[htbp]
\centering
\small
\renewcommand{\arraystretch}{1.18}
\setlength{\tabcolsep}{4pt}
\caption{{\small \textbf{Rainbow spanning configurations in randomly perturbed graphs.}
Every row assumes $\delta(G_0)\geq\delta n$ and asserts existence
asymptotically almost surely in the uniformly coloured union
$G_0\cup \mathbb G(n,p)$. A prescribed tree has maximum degree at most
the fixed constant $\Delta$; an unspecified spanning tree need not satisfy
this bound. Palette sizes are rounded to integers.}}
\label{tab:perturbed-rainbow-literature}
\begin{tabular}{@{}>{\raggedright\arraybackslash}p{.27\textwidth}>{\raggedright\arraybackslash}p{.17\textwidth}>{\raggedright\arraybackslash}p{.24\textwidth}>{\raggedright\arraybackslash}p{.25\textwidth}@{}}
\hline
\textbf{Authors and reference} & \textbf{Palette size} &
\textbf{Configuration} & \textbf{Parameters}\\[4pt]
\hline
Anastos--Frieze~\cite{AnastosFrieze2019Perturbed}
& $K_\delta n$
& Hamilton cycle
& $p\geq C(\delta)/n$; $K_\delta=120-20\log\delta$\\[4pt]
Aigner-Horev--Hefetz~\cite{AignerHorevHefetz2021Perturbed}
& $(1+\varepsilon)n$
& Hamilton cycle
& Fixed $\varepsilon>0$; $p\geq C(\delta,\varepsilon)/n$\\[4pt]
Aigner-Horev--Hefetz--Lahiri~\cite{AignerHorevHefetzLahiri2023Trees}
& $(1+\varepsilon)n$
& Prescribed $T \in \cT(n,\Delta)$
& Fixed $\varepsilon>0$; $p=\omega(1/n)$\\[4pt]
Aigner-Horev--Hefetz--Lahiri~\cite{AignerHorevHefetzLahiri2023Trees}
& $n-1$
& Unspecified spanning tree
& $p=\omega(n^{-2})$\\[4pt]
Katsamaktsis--Letzter--Sgueglia~\cite{KatsamaktsisLetzterSgueglia2024Hamilton}
& $n$
& Hamilton cycle
& $p\geq C(\delta)/n$\\[4pt]
Katsamaktsis--Letzter--Sgueglia~\cite{KatsamaktsisLetzterSgueglia2025Perturbed}
& $n-1$
& Prescribed $T \in \cT(n,\Delta)$
& $p\geq C(\delta,\Delta)/n$\\[4pt]
\hline
\end{tabular}
\end{table}

\medskip
\noindent\textbf{Spread measures and rainbow thresholds.}
Roughly put, {\sl spread measures} provide a common framework for studying copies of a target configuration by controlling the probability that a randomly chosen copy contains any prescribed set of edges. For a finite set $\Omega$ and a subset $S\subseteq\Omega$, let
$\langle S\rangle = \{A\subseteq\Omega:S\subseteq A\}$.

\begin{definition}[Spread measure]
\label{def:spread}
Let $\Omega$ be finite and let $q>0$. A probability measure $\mu$ on
$2^\Omega$ is \emph{$q$-spread} if
$$
 \mu(\langle S\rangle)\leq q^{|S|}
 \qquad\text{holds for every }S\subseteq\Omega.
$$
A family $\cA\subseteq2^\Omega$ \emph{supports} such a measure if
$\mu(\cA)=1$.
\end{definition}

Thus, the probability of including any prescribed set of $s$ elements is at most $q^s$.
For graph embeddings in a host graph $G$, the ground set is $E(G)$ and the members of $\cA$ are edge-sets of copies of the target configuration; a smaller $q$ expresses
stronger control over joint edge-inclusion probabilities.

The notion of spread connects the distribution of configurations in a
fixed host with their survival under random sampling/percolation. Spread measures were introduced by Talagrand~\cite[Section~6]{Talagrand2010} and their development is closely associated with the (abstract) expectation threshold conjecture by Kahn and Kalai~\cite{KahnKalai2007}, its relaxation by Talagrand~\cite[Section~8]{Talagrand2010},
the breakthrough work of Alweiss, Lovett, Wu and
Zhang on the sunflower conjecture~\cite{AlweissLovettWuZhang2021}, and the work by Frankston, Kahn, Narayanan and Park~\cite{FrankstonKahnNarayananPark2021} who proved Talagrand's expectation threshold conjecture, and the subsequent resolution of the expectation threshold conjecture of Kahn and Kalai by Park and Pham~\cite{ParkPham2024}. The result in~\cite{FrankstonKahnNarayananPark2021} shows that, if a family of sets of size at most $h$ supports a $q$-spread measure, then the randomly retained set a.a.s. contains a member of the family when each ground-set element is retained independently at rate $Cq\log h$, provided this rate is at most one. A (full) rainbow version for uniform spread measures was recently obtained by Han and Yuan~\cite{HanYuan2025}.
 We use spread measures to reduce the rainbow embedding problem in coloured percolated bijumbled graphs to constructing a measure on uncoloured configurations with suitable bounds on the probability that a randomly chosen configuration contains any prescribed set of edges.

For comparison with the hypergraph terminology, an $h$-bounded
multihypergraph is a family of sets of size at most $h$, with repeated
members allowed; it is $\kappa$-spread if its uniform edge measure is
$\kappa^{-1}$-spread. Table~\ref{tab:spread-thresholds} records the
rainbow results most directly relevant to our results; all results are with respect to uniform colourings. Here, $N$ is the size of the ground set, $\varepsilon>0$ is fixed, and $\rho$ is the retention probability of edges of the host.

\begin{table}[htbp]
\centering
\small
\setlength{\tabcolsep}{4pt}
\renewcommand{\arraystretch}{1.2}
\caption{{\small \textbf{Spread measures and exact-palette rainbow containment.}}}
\label{tab:spread-thresholds}
\small
\begin{tabularx}{\textwidth}{@{}>{\raggedright\arraybackslash}p{.24\textwidth}>{\raggedright\arraybackslash}p{.10\textwidth}>{\raggedright\arraybackslash}p{.23\textwidth}>{\raggedright\arraybackslash}X@{}}
\toprule
\textbf{Authors and reference} & \textbf{Palette size} & \textbf{Configuration} & \textbf{Parameters} \\[4pt]
\midrule
Bell--Frieze--Marbach~\cite{BellFriezeMarbach2024}
& $h$
& Edge of an $h$-bounded multihypergraph
& $\kappa$-spread; $\kappa=\Omega(h)$,
  $N\leq\kappa^2h/(\log h)^5$;
  a uniform $m$-set with $m\geq C N\log h/\kappa$ \\[4pt]
Han--Yuan~\cite{HanYuan2025}
& $h$
& Edge of an $h$-bounded multihypergraph
& $\kappa$-spread; independent retention
  $\rho\geq C\log h/\kappa$ \\[4pt]
Han--Yuan~\cite{HanYuan2025}
& $n-1$
& Prescribed tree $T \in \cT(n,\Delta)$
& $\delta(G)\geq(1/2+\varepsilon)n$;
  $\rho=C\log n/n$ \\[4pt]
Han--Yuan~\cite{HanYuan2025}
& $n/k$
& Perfect matching in a $k$-uniform hypergraph
& Minimum $\ell$-degree above its asymptotic Dirac threshold;
  $\rho=C\log n/n^{k-1}$ \\[4pt]
\bottomrule
\end{tabularx}
\end{table}

The first row already permits an exact palette; its restrictions concern
the spread parameter and the ground-set size.
The second row removes both restrictions while preserving the
$(\log h)/\kappa$ sampling rate. Thus, the advance of
Han--Yuan~\cite[Theorem~3]{HanYuan2025} over
Bell--Frieze--Marbach~\cite[Theorem~2]{BellFriezeMarbach2024} is a general
rainbow spread theorem. Han--Yuan also establish a rainbow analogue of
Spiro's theorem~\cite{Spiro2023}, allowing smaller sampling rates
under stronger overlap conditions.

The last two rows concern randomly coloured sparsifications of fixed
dense hosts. For fixed $1\leq\ell<k$, let $\delta^+_{\ell,k}$ be the
infimum of the real numbers $a$ such that, for every $\varepsilon>0$
and all sufficiently large multiples $n$ of $k$, every $n$-vertex
$k$-uniform hypergraph of minimum $\ell$-degree at least
$(a+\varepsilon)\binom{n}{k-\ell}$ admits a perfect matching.
In the last row, the hypothesis is
$\delta_\ell(H)\geq(\delta^+_{\ell,k}+\varepsilon)
\binom{n}{k-\ell}$. The tree and matching assertions combine the
rainbow theorem with spread constructions of
Pham, Sah, Sawhney and Simkin~\cite{PhamSahSawhneySimkin2022} and,
for hypergraph matchings, the independent work of
Kang, Kelly, K\"uhn, Osthus and
Pfenninger~\cite{KangKellyKuhnOsthusPfenninger2022}.

The removal of the condition $\kappa=\Omega(h)$ is essential for the
measures used here. For matchings and Hamilton cycles we construct
$q$-spread measures with $q$ of order $\log n/(pn)$;
the resulting value $\kappa=1/q$ has order $pn/\log n$, whereas
$h$ has order $n$. In particular, when $p=o(1)$, these parameters do
not meet the additional hypothesis of Bell--Frieze--Marbach.

We employ spread measures in our work here due to their ability to serve as deterministic certificates ensuring high probability random containment~\cite{BellFriezeMarbach2024,FrankstonKahnNarayananPark2021,HanYuan2025} in various structures of interest.

\FloatBarrier
\medskip
\noindent\textbf{Pseudorandom and expander hosts.}
The deterministic theory of spanning embeddings provides the structural
input for our rainbow results. Its development encompasses Hamilton
cycles and their powers, almost spanning and spanning bounded-degree
trees, prescribed $2$-factors, and clique factors in spectral and bijumbled
graphs; see~\cite{AllenBottcherHanKohayakawaPerson2017,AlonKrivelevichSudakov2007,BaloghCsabaPeiSamotij2010, FerberHanMaoVershynin2024, GlockMunhaCorreiaSudakov2023,HanKohayakawaMorrisPerson2019,HanKohayakawaMorrisPerson2021,HanKohayakawaPerson2021,HefetzKrivelevichSzabo2009,HydeMorrisonMuyesserPavezSigne2025, JohannsenKrivelevichSamotij2013,KrivelevichSudakov2003,KrivelevichSudakovSzabo2004,Nenadov2019}
and our companion work~\cite{AignerHorevHefetzPersonTrushkin2026}
for a fuller account. We state below the three deterministic embedding
theorems used in our proofs; their respective hypotheses underlie the
different bijumbledness scales in our conclusions.

\begin{definition}
\label{def:external-neighbourhood}
For a graph $G$ and a set $X\subseteq V(G)$, the \emph{external
neighbourhood} of $X$ in $G$ is
$$
 \Gamma_G(X):=\{v\in V(G)\setminus X:
                 \text{$v$ has a neighbour in $X$}\}.
$$
\end{definition}

\begin{definition}
\label{def:expander}
An $n$-vertex graph $G$ is said to form a 
\emph{$C$-expander}, where $C > 0$, if both of the following assertions hold.
\begin{enumerate}[label=\textup{(\alph*)}]
\item [\emph{(E1)}] $|\Gamma_G(X)| \geq C|X|$ for every
      $X \subseteq V(G)$ of size $|X| < n/(2C)$.
\item [\emph{(E2)}] There is an edge of $G$ between every two disjoint sets
      $X,Y \subseteq V(G)$ of sizes $|X|,|Y|\geq n/(2C)$.
\end{enumerate}
\end{definition}

The following breakthrough result of Dragani\'c, Montgomery, Munh\'a Correia, Pokrovskiy, and Sudakov~\cite{DraganicMontgomeryMunhaCorreiaPokrovskiySudakov2024}
is the deterministic input for both of our Hamilton cycle results; it asserts that constant expansion suffices for Hamiltonicity.

\begin{theorem}
\emph{~\cite[Theorem~1.4]{DraganicMontgomeryMunhaCorreiaPokrovskiySudakov2024}}
\label{thm:expander-hamilton}
For every sufficiently large constant $C$, every $C$-expander is
Hamiltonian.
\end{theorem}

For spanning trees we use the stronger expansion requirement of
Han and Yang~\cite{HanYang2022}. A graph is called \emph{$\cT(n,\Delta)$-universal}
 if it contains a copy of every member of $\cT(n,\Delta)$.

\begin{theorem}
\emph{~\cite[Theorem~1.11]{HanYang2022}}
\label{thm:han-yang}
There exists a constant $n_0$ such that for all integers
$n \geq n_0$ and $\Delta \geq 2$, every $C$-expander satisfying $C\geq\Delta^{5\sqrt{\log n}}$ is $\cT(n,\Delta)$-universal.
\end{theorem}

For comparison with our rainbow tree embedding results, we also record the following theorem of Han and Yang~\cite{HanYang2022}


\begin{theorem}[{\cite[Theorem~1.5]{HanYang2022}}]
\label{thm:han-yang-pseudo}
Let $\Delta \geq 2$ be an integer and let $n$ be a sufficiently large integer. Then, any $n$-vertex $(p,\beta)$-bijumbled graph $F$ satisfying
$$
 \delta(F)\geq4\sqrt{\beta pn}
 \qquad \textrm{and} \qquad \beta\leq\frac{pn}{4\Delta^{5\sqrt{\log n}}},
$$
is $\cT(n,\Delta)$-universal.
\end{theorem}

For $(n,d,\lambda)$-graphs which, as previously noted, are $d$-regular $(d/n,\lambda)$-bijumbled graphs, Hyde, Morrison, M\"uyesser and Pavez-Sign\'e~\cite{HydeMorrisonMuyesserPavezSigne2025} attain $\cT(n,\Delta)$-universality with $\lambda\leq pn/[C_\Delta \log^3 n]$ which is better than the bound attained in Theorem~\ref{thm:han-yang-pseudo}. To the best of our  knowledge, the Hyde--Morrison--M\"uyesser--Pavez-Sign\'e bound on the discrepancy parameter is not known to hold for the more general class of bijumbled graphs.  

\medskip

Our proofs pertaining to clique factors rely on the
following result of Morris~\cite{Morris2025}.

\begin{theorem}
\emph{~\cite[Theorem~1.4]{Morris2025}}
\label{thm:morris}
For every integer $k\geq3$ and every $\zeta>0$, there exists
$\eta\coloneqq\eta(k,\zeta)>0$ such that every $n$-vertex
$(p,\beta)$-bijumbled graph $G$ with $p>0$ satisfying
$$
 \delta(G)\geq\zeta pn,\qquad
 \beta\leq\eta p^{k-1}n,\qquad \textrm{and} \qquad k\mid n
$$
admits a $K_k$-factor.
\end{theorem}

Chen--Chen--Han--Zhao~\cite[Theorems~1.4 and~3.1]{ChenChenHanZhao2025}
establish uncoloured robustness for bijumbled graphs.
Given a graph $G$, write $G_\rho$ to denote its percolation at rate $\rho$, indicating the random spanning subgraph of $G$ obtained by retaining each edge independently with probability $\rho$. For fixed
$\alpha,\eta\in(0,1]$, sufficiently small
$\varepsilon := \varepsilon(\alpha,\eta)>0$, and an $n$-vertex
$(p,\beta)$-bijumbled graph $G$ satisfying
$$
 \delta(G)\geq\alpha pn,\qquad
 \beta\leq\varepsilon pn,\qquad \textrm{and} \qquad
 \frac{(1+\eta)\log n}{\alpha pn} \leq \rho \leq 1,
$$
their expansion argument implies that $G_\rho$ is a.a.s. Hamiltonian; for even $n$ it thus also admits a perfect matching. Their balanced bipartite perfect-matching counterpart~\cite[Theorem~1.5]{ChenChenHanZhao2025} uses the same conditions, with $n$ denoting the size of each part.

For triangle factors, Theorem~1.6 in~\cite{ChenChenHanZhao2025} applies to $(n,d,\lambda)$-graphs with
$$
 3\mid n,\qquad
 d\geq Cn^{5/6}(\log n)^{1/2},\qquad \textrm{and} \qquad
 \lambda\leq\varepsilon d^2/n,
$$
where $\varepsilon>0$ is sufficiently small and
$C=C(\varepsilon)$ sufficiently large; this theorem then asserts that  $G_\rho$ a.a.s. admits a triangle factor whenever
$\rho d\gg n^{1/3}(\log n)^{1/3}$.
Their proof combines iterative absorption, an $O(n/d^3)$-spread
measure on triangle factors, represented as collections of triangles,
and a sparse coupling lemma~\cite[Lemma~5.2 and Section~6]{ChenChenHanZhao2025}. The ground-set elements of this measure are whole triangles; in contrast, our rainbow problem assigns colours to individual graph edges.

\subsection{Our rainbow embedding results with sublinear palette surplus}\label{sec:sublinear}

Our companion work~\cite{AignerHorevHefetzPersonTrushkin2026}
employs a McDiarmid-type~\cite{McDiarmid1981} palette coupling, stated in Lemma~\ref{lem:coupling} below, in order 
to prove a general transference result that yields rainbow spanning configurations in uniformly coloured hosts with a small positive linear
palette surplus. Our first sequence of results reduces the surplus to $o(n)$ for various specific configurations of interest, while retaining the discrepancy scales of the deterministic inputs. These are 
$\beta = O(pn)$ for perfect matchings and Hamilton cycles, and
$\beta = O(pn/L_{n,\Delta})$ for prescribed spanning trees.
We prove all three results using quantitative estimates for the
uncoloured percolations selected by the coupling lemma, namely Lemma~\ref{lem:coupling}.



At Morris' discrepancy scale, the same coupling and sparsification
argument does not establish a sublinear surplus $K_k$-factor result
for $k\geq3$. Section~\ref{sec:limitations} identifies the quantitative
restriction of this argument and explains how stronger host assumptions
can overcome it.

\medskip
Our first sequence of results reads as follows. 

\begin{theorem}[Rainbow perfect matchings with sublinear palette surplus]
\label{thm:sublinear-surplus-perfect-matching}
For every $\alpha,\eta>0$, there exists a constant
$c:=c(\alpha,\eta)>0$ such that the following holds. 
Let $n$ range over the even positive integers, and let
$0<p:=p(n)\leq1$ satisfy $pn=\omega(\log n)$.
Let $G$ be an $n$-vertex $(p,\beta)$-bijumbled graph satisfying
$\delta(G)\geq\alpha pn$ and $\beta\leq cpn$.
Then, a $Q$-uniform colouring of $G$, where
$$
 Q\geq\frac n2+
 \left(\frac2\alpha+\eta+o(1)\right)\frac{\log n}{p},
$$
a.a.s. admits a rainbow perfect matching.
\end{theorem}

\begin{theorem}[Rainbow prescribed spanning trees with sublinear palette surplus]
\label{thm:sublinear-surplus-tree}
For every $\alpha>0$ and every integer $\Delta\geq2$, there exist
constants $c:=c(\alpha)>0$ and $\kappa \coloneqq \kappa(\alpha) >0$ such that the following holds. Let $L_{n,\Delta} = \Delta^{5\sqrt{\log n}}$ and let $G$ be an $n$-vertex $(p,\beta)$-bijumbled graph satisfying 
$$
\delta(G) \geq \alpha pn \qquad \textrm{and} \qquad \beta\leq c\frac{pn}{4L_{n,\Delta}}.
$$
If $G$ is $Q$-uniformly coloured, where 
$$
 Q \geq (n-1) + \left(\frac{1}{\kappa} +o(1)\right)\frac{L_{n,\Delta} \log n}{p},
$$
then
$$
    \sup_{T\in\cT(n,\Delta)}
    \PP\left[
        G\text{ admits no rainbow copy of }T
    \right]
    =
    o(1).
$$
\end{theorem}

\begin{remark}
\label{rem:not-simultaneous-trees-short}
In Theorem~\ref{thm:sublinear-surplus-tree} (and similarly in Theorem~\ref{thm:exact-trees} below) the failure probability tends to 0 uniformly over $\cT(n,\Delta)$; this does not imply simultaneous rainbow universality.
\end{remark}

\begin{theorem}[Rainbow Hamilton cycles with sublinear palette surplus]
\label{thm:sublinear-surplus-hamilton}
For every $\alpha>0$, there exist constants
$A\coloneqq A(\alpha)>0$ and $c\coloneqq c(\alpha)>0$
such that the following holds. Let $0<p:=p(n)\leq1$ satisfy
$pn=\omega(\log n)$, and let $G$ be an $n$-vertex
$(p,\beta)$-bijumbled graph satisfying $\delta(G)\geq\alpha pn$
and $\beta\leq cpn$. Then, a $Q$-uniform colouring of $G$ using 
$$
Q \geq n + (A+o(1))\frac{\log n}{p}
$$
colours a.a.s. admits a rainbow Hamilton cycle. 
\end{theorem}

\subsection{Our exact-palette percolated rainbow embedding results}\label{sec:exact}

Our second sequence of results uses exact palettes: $n/2$ colours for a perfect matching, $(k-1)n/2$ for a $K_k$-factor, $n-1$ for a prescribed spanning tree, and $n$ for a Hamilton cycle. Each conclusion holds in a uniformly coloured $(p, \beta)$-bijumbled graph that is then further percolated with retention probability $\rho$.


For matchings and Hamilton cycles, the retained density condition is
$\rho pn\geq C(\log n)^2$, whilst the host discrepancy remains
$\beta\leq\gamma pn$. For prescribed bounded-degree trees, the corresponding condition is $\rho pn\geq C L_{n,\Delta}(\log n)^2$, with
$\beta\leq\gamma pn/L_{n,\Delta}$. Clique factors require the additional
factor $(\rho/\log n)^{k-2}$ in the bound on $\beta$; its role is
explained after the statement of Theorem~\ref{thm:spread-clique-factor}.
These are sufficient quantitative conditions, and no optimality of
their logarithmic factors is claimed.

\medskip

Our second sequence of results reads as follows. 

\begin{theorem}[Percolated rainbow perfect matchings]
\label{thm:exact-matching}
For every $c>0$, there exist constants $C\coloneqq C(c)>0$ and
$\gamma\coloneqq\gamma(c)>0$ such that the following holds.
Let $G$ be an $n$-vertex $(p,\beta)$-bijumbled graph and let $\rho \coloneqq \rho(n) \in(0,1]$ satisfy 
$$
 \delta(G)\geq cpn,\qquad
 \beta\leq\gamma pn,\qquad
 \rho pn\geq C(\log n)^2,\qquad \textrm{and} \qquad 2\mid n.
$$
Then, following an $[n/2]$-uniform colouring of $G$, the percolation $G_\rho$ a.a.s. admits a rainbow perfect matching.
\end{theorem}

\begin{theorem}[Percolated rainbow clique factors]
\label{thm:exact-clique-factor}
For every integer $k\geq3$ and every $\alpha\in(0,1]$, there exist
constants $C\coloneqq C(k,\alpha)>0$ and
$\gamma\coloneqq\gamma(k,\alpha)>0$ such that the following holds.
Let $G$ be
an $n$-vertex $(p,\beta)$-bijumbled graph and let $\rho \coloneqq \rho(n) \in (0,1]$ satisfy 
$$
 \delta(G)\geq\alpha pn,
 \qquad
 \beta\leq\gamma
       \left(\frac{\rho}{\log n}\right)^{k-2}p^{k-1}n,
 \qquad
 \rho p\geq Cn^{-1/(2k-3)}\log n,
 \qquad \textrm{and} \qquad k\mid n.
$$
Then, following a 
$[(k-1)n/2]$-uniform colouring of $G$, the percolation $G_\rho$ a.a.s. admits a rainbow $K_k$-factor.
\end{theorem}

The additional factor $(\rho/\log n)^{k-2}$ in the upper bound on $\beta$ appearing in Theorem~\ref{thm:exact-clique-factor} is absent from Morris'
uncoloured theorem. At $\rho=1$ this is a polylogarithmic strengthening
of the host discrepancy assumption; an explanation follows the statement of 
Theorem~\ref{thm:spread-clique-factor}.

\begin{theorem}[Percolated rainbow prescribed spanning trees]
\label{thm:exact-trees}
For every $\alpha\in(0,1]$, there exist constants
$C\coloneqq C(\alpha)>0$ and $\gamma\coloneqq\gamma(\alpha)>0$
such that, for every fixed integer $\Delta\geq2$, the following holds.
Let $L_{n,\Delta} = \Delta^{5\sqrt{\log n}}$, let $G$ be an $n$-vertex $(p,\beta)$-bijumbled graph, and let $\rho := \rho(n) \in(0,1]$ satisfy 
$$
 \delta(G)\geq\alpha pn,\qquad
 \beta\leq\gamma\frac{pn}{L_{n,\Delta}},\qquad \textrm{and} \qquad
 \rho pn\geq C L_{n,\Delta}(\log n)^2.
$$
Then, following an $[n-1]$-uniform colouring of $G$,
$$
 \sup_{T\in\cT(n,\Delta)}
 \PP[\text{$G_\rho$ admits no rainbow copy of $T$}]=o(1).
$$
\end{theorem}

Up to constants, the bounds on $\beta$ appearing in Theorems~\ref{thm:han-yang-pseudo} and ~\ref{thm:exact-trees} coincide. This (asymptotic) bound on the discrepancy implicitly imposes the density requirement $pn = \Omega(L_{n,\Delta}^2)$ (see Observation~\ref{obs:density-floor} and its proof below). In particular, in the setting of Theorem~\ref{thm:exact-trees} with percolation rate $\rho =1$, we have $pn \geq \frac{\alpha}{4 \gamma^2}L_{n,\Delta}^2$ for sufficiently large $n$. Consequently, at $\rho =1$, the explicit density requirement of $pn = \Omega (L_{n,\Delta} (\log n)^2)$, imposed in Theorem~\ref{thm:exact-trees}, is already subsumed by the implicit density threshold imposed by the aforementioned bound on the discrepancy parameter. On the other hand, the minimum degree condition imposed in Theorem~\ref{thm:exact-trees} is stronger than the one imposed in Theorem~\ref{thm:han-yang-pseudo}. 


\begin{theorem}[Percolated rainbow Hamilton cycles]
\label{thm:exact-hamilton}
For every $\alpha\in(0,1]$, there exist constants
$C\coloneqq C(\alpha)>0$ and $\gamma\coloneqq\gamma(\alpha)>0$
such that the following holds.
Let $G$ be an $n$-vertex $(p,\beta)$-bijumbled graph and let
$\rho \coloneqq \rho(n) \in(0,1]$ satisfy 
$$
 \delta(G)\geq\alpha pn,\qquad
 \beta\leq\gamma pn, \qquad \textrm{and} \qquad
 \rho pn\geq C(\log n)^2.
$$
Then, following an $[n]$-uniform colouring of $G$, the percolation $G_\rho$ a.a.s. admits a rainbow Hamilton cycle.
\end{theorem}

\begin{remark}
The sizes of the palettes in Theorems~\ref{thm:exact-matching} and~\ref{thm:exact-hamilton} are different; hence, the latter theorem does not imply the former. 
\end{remark}

\subsubsection{Comparison with the uncoloured results of Chen--Chen--Han--Zhao~\cite{ChenChenHanZhao2025}}

As mentioned above, Chen--Chen--Han--Zhao~\cite{ChenChenHanZhao2025} study the statistical robustness of spanning configurations in $(n,d,\lambda)$-graphs in the colourless setting. In this section, we perform a parameter comparison between our exact-palette rainbow results and their results.  
Table~\ref{tab:chen-rainbow-matching-hamilton} starts this comparison by considering the results pertaining to perfect matchings and Hamiltonicity; note that, like our results, these results from~\cite{ChenChenHanZhao2025} cover the more general class of $(p, \beta)$-bijumbled graphs. 

\begin{table}[htbp]
\centering
\small
\caption{{\small \textbf{Uncoloured and exact-palette rainbow robustness.}
The host is an $n$-vertex $(p,\beta)$-bijumbled graph and
$\rho\in(0,1]$. All conclusions hold asymptotically almost surely.
Here, $\alpha\in(0,1]$ is fixed, $\varepsilon,\gamma>0$ are sufficiently
small, and $C>0$ is sufficiently large; constants may differ between
columns and depend on $\alpha$. Perfect matchings require $2\mid n$.}}
\label{tab:chen-rainbow-matching-hamilton}
\setlength{\tabcolsep}{5pt}
\renewcommand{\arraystretch}{1.15}
\begin{tabularx}{\textwidth}{@{}
 >{\raggedright\arraybackslash}p{0.22\textwidth}
 >{\raggedright\arraybackslash}X
 >{\raggedright\arraybackslash}X@{}}
\toprule
&
\textbf{Chen--Chen--Han--Zhao}\newline
~\cite[Theorem~3.1]{ChenChenHanZhao2025}
&
\textbf{Our exact-palette results}\newline
Theorems~\ref{thm:exact-matching}
and~\ref{thm:exact-hamilton}
\\
\midrule
Host hypotheses
& $\delta(G)\geq\alpha pn$,\newline $\beta\leq\varepsilon pn$
& $\delta(G)\geq\alpha pn$,\newline $\beta\leq\gamma pn$
\\[4pt]
Sufficient retained degree scale
& $\rho pn\geq C\log n$
& $\rho pn\geq C(\log n)^2$
\\[4pt]
Conclusion
& Perfect matching or Hamilton cycle
& Rainbow perfect matching or rainbow Hamilton cycle
\\[4pt]
Palette size
& Uncoloured
& $n/2$ for a perfect matching;\newline $n$ for a Hamilton cycle
\\
\bottomrule
\end{tabularx}
\end{table}

Theorems~\ref{thm:exact-matching} and~\ref{thm:exact-hamilton}
extend the uncoloured robustness results of
Chen--Chen--Han--Zhao~\cite[Theorems~1.4 and~3.1]{ChenChenHanZhao2025}
to independent uniform edge colourings with exact palettes, at the
cost of an additional logarithmic factor in the sufficient retention
rate. The minimum-degree and bijumbledness assumptions retain the same
scales, with suitable choices of constants.
\FloatBarrier

\medskip
Proceeding to clique factors, Theorem~\ref{thm:exact-clique-factor} covers every fixed $k\geq3$
with exactly $(k-1)n/2$ colours. Its comparison with the uncoloured
triangle-factor result of
Chen--Chen--Han--Zhao~\cite[Theorem~1.6]{ChenChenHanZhao2025}
requires attention to the different host assumptions.
For $3\mid n$, their result concerns $(n,d,\lambda)$-graphs satisfying
$$
 d\geq C_0n^{5/6}(\log n)^{1/2}
 \qquad \textrm{and} \qquad
 \lambda\leq\varepsilon\frac{d^2}{n},
$$
where $\varepsilon>0$ is sufficiently small and
$C_0=C_0(\varepsilon)>0$ is sufficiently large.
The percolation $G_\rho$ a.a.s. admits a triangle factor whenever
$$
 \rho d\gg n^{1/3}(\log n)^{1/3}.
$$
On these regular hosts, specialising
Theorem~\ref{thm:exact-clique-factor} to $k=3$, $p=d/n$,
$\beta=\lambda$, and $\alpha=1$, instead requires
$$
 \lambda\leq\gamma\frac{\rho}{\log n}\cdot\frac{d^2}{n}
 \qquad \textrm{and} \qquad
 \rho d\geq Cn^{2/3}\log n,
$$
where $\gamma,C>0$ are the constants supplied by that theorem;
the resulting triangle factor is rainbow with exactly $n$ colours.
Thus, on the common class of spectral hosts, our sufficient
discrepancy bound is smaller by a factor of order $\rho/\log n$,
and our sufficient retention requirement is stronger. Apart from obtaining clique factors that are rainbow, we handle every fixed clique size and the more general class of bijumbled graphs. Moreover, we consider host graphs with a smaller minimum degree.
 

\subsection{A palette transference principle}

The four exact-palette applications follow from
Theorem~\ref{thm:spread-transference}, which combines a containment
estimate with the Han--Yuan~\cite{HanYuan2025} rainbow threshold theorem stated as Theorem~\ref{thm:han-yuan}.

It is useful to see the core message of Theorem~\ref{thm:spread-transference} in graph terminology prior to its
formal statement. If the edge-sets of specified spanning configurations
in a graph $G$ support a sufficiently spread measure, then this
\emph{deterministic} property ensures a form of \emph{statistical
robustness}: after an independent uniform colouring, typical
percolations of $G$ retain rainbow copies of those configurations,
even when the palette contains exactly the required number of colours.

\medskip
 For a finite set $\Omega$ and $\rho\in[0,1]$, the random set $\Omega_\rho$ retains each element of $\Omega$ independently with probability $\rho$.

\begin{theorem}[Exact-palette spread transference with percolation]
\label{thm:spread-transference}
For every positive integer $n$, let $\Omega_n$ be a finite set and let
$\cA_n \subseteq \binom{\Omega_n}{h_n}$, where $h_n\to\infty$. Let $\vartheta_n \in (0,1]$
and let 
$$
 a_n = \PP[\text{$(\Omega_n)_{\vartheta_n}$ contains a member of $\cA_n$}].
$$
If $a_n>0$ for every $n$ and
$
 q_n:=\frac{\vartheta_n}{a_n} \leq 1/2
$
holds for all sufficiently
large $n$, then $\cA_n$ supports a $q_n$-spread  measure. 

Moreover, there exists an absolute constant $C_{\rm tr}>0$ such that, upon colouring the members of $\Omega_n$ independently and uniformly using the palette $[h_n]$, 
$$
\PP[(\Omega_n)_{\rho_n}\; \text{admits a rainbow member of}\; \cA_n] = 1 - o(1)
$$
holds whenever $\rho_n\in(0,1]$ satisfies
$
 \rho_n\geq
 2C_{\rm tr}\cdot q_n \log h_n
$,
where the probability is over both the colouring and the
$\rho_n$-percolation. 

In particular, there exists an absolute constant $c_{\rm sp}>0$
such that, whenever $q_n\log h_n\leq c_{\rm sp}$ for all sufficiently large $n$, an independent uniform colouring of
$\Omega_n$ from $[h_n]$ a.a.s. admits a rainbow member of $\cA_n$.
\end{theorem}

\begin{remark}
The auxiliary retention rate $\vartheta_n$ is only used in order to construct
 the spread measure from an uncoloured containment experiment.
The final percolation at rate $\rho_n$ is sampled directly from
$\Omega_n$, independently of the colouring. These are separate uses
of sampling, not successive percolations; the final retention rate is
$\rho_n$, not $\vartheta_n\rho_n$.
\end{remark}

Theorem~\ref{thm:spread-transference} isolates the following useful
argument. If an auxiliary uncoloured $\vartheta$-percolation contains
a target with probability $a>0$, conditional selection provides a
$(\vartheta/a)$-spread measure. This deterministic certificate then
supports exact-palette rainbow containment at final retention rate
$\rho\geq2C_{\rm tr}(\vartheta/a)\log h$. Thus, one can verify the
rainbow conclusion through uncoloured containment estimates and
deterministic embedding theorems. In our applications $a=1-o(1)$,
so conditioning changes the spread parameter by a factor $1+o(1)$.


\subsubsection{Spread form of our rainbow embedding results} Using the notion of spread measures and Theorem~\ref{thm:spread-transference}, we now state strengthenings of the results of Section~\ref{sec:exact}; these have the same percolation and rainbow conclusions but with the additional claim that certain spread measures exist. 

\begin{theorem}[Spread form of Theorem~\ref{thm:exact-matching}]
\label{thm:spread-matching}
For every $c>0$, there exist constants $C\coloneqq C(c)>0$ and
$\gamma\coloneqq\gamma(c)>0$ such that the following holds.
Let $G$ be an $n$-vertex $(p,\beta)$-bijumbled graph and let
$\rho \coloneqq \rho(n)\in(0,1]$ satisfy 
$$
 \delta(G)\geq cpn,\quad
 \beta\leq\gamma pn,\quad
 \rho pn\geq C(\log n)^2, \quad \textrm{and} \quad 2\mid n.
$$
Then, the family of perfect matchings of $G$ supports a $C\log n/(pn)$-spread measure.
Moreover, following an $[n/2]$-uniform colouring of $G$, the percolation $G_\rho$ a.a.s. admits a rainbow perfect matching.
\end{theorem}

\begin{theorem}[Spread form of Theorem~\ref{thm:exact-clique-factor}]
\label{thm:spread-clique-factor}
For every integer $k\geq3$ and every $\alpha\in(0,1]$, there exist
constants $C\coloneqq C(k,\alpha)>0$,
$\gamma\coloneqq\gamma(k,\alpha)>0$, and
$\sigma\coloneqq\sigma(k,\alpha)>0$ such that the following holds.
Let $G$ be an $n$-vertex $(p,\beta)$-bijumbled graph and let
$\rho \coloneqq \rho(n) \in(0,1]$ satisfy 
$$
 \delta(G)\geq\alpha pn,\quad
 \beta\leq\gamma
       \left(\frac{\rho}{\log n}\right)^{k-2}p^{k-1}n,
 \quad
 \rho p\geq Cn^{-1/(2k-3)}\log n,\quad \textrm{and} \quad k\mid n.
$$
Then, the family of $K_k$-factors of $G$ supports a $(\sigma\rho/\log n)$-spread measure.
Moreover, following a $[(k-1)n/2]$-uniform colouring of $G$, the percolation $G_\rho$ a.a.s. admits a rainbow $K_k$-factor.
\end{theorem}

The additional factor $(\rho/\log n)^{k-2}$ arises from the auxiliary
retention rate $\vartheta=b\rho/\log n$, where $b > 0$ is some constant, used to construct the spread measure. The discrepancy estimate after this sampling contains the term $\vartheta\beta$, whereas Morris' theorem at density $\vartheta p$ permits discrepancy of order
$(\vartheta p)^{k-1}n$. Comparing these terms requires the original
bound $\beta=O(\vartheta^{k-2}p^{k-1}n)$, which explains the stated
factor. Taking $\rho=1$ removes the final percolation but leaves this
auxiliary sampling in the proof, and hence the polylogarithmic loss.
This is a sufficient condition for the method to apply, not a claimed necessary condition
for rainbow clique factors.

\begin{theorem}[Spread form of Theorem~\ref{thm:exact-trees}]
\label{thm:spread-trees}
For every $\alpha\in(0,1]$, there exist constants
$C\coloneqq C(\alpha)>0$ and $\gamma\coloneqq\gamma(\alpha)>0$
such that, for every fixed integer $\Delta\geq2$, the following
holds. Let $L_{n,\Delta} = \Delta^{5\sqrt{\log n}}$, let $G$ be an $n$-vertex $(p,\beta)$-bijumbled graph, and let $\rho \coloneqq \rho(n) \in(0,1]$ satisfy 
$$
 \delta(G)\geq\alpha pn,\quad
 \beta\leq\gamma\frac{pn}{L_{n,\Delta}},\quad \textrm{and} \quad
 \rho pn\geq C L_{n,\Delta}(\log n)^2.
$$
Then, for each prescribed $T\in\cT(n,\Delta)$, the family of
copies of $T$ in $G$ supports
a $q_T$-spread measure satisfying
\begin{equation}\label{eq:q_T}
 q_T\leq C\frac{L_{n,\Delta}\log n}{pn}
      =o\left(\frac1{\log n}\right),
\end{equation}
uniformly over $T$. Moreover, following an $[n-1]$-uniform colouring of $G$, it holds that 
$$
 \sup_{T\in\cT(n,\Delta)}
 \PP[\text{$G_\rho$ admits no rainbow copy of $T$}]=o(1).
$$
\end{theorem}

\begin{theorem}[Spread form of Theorem~\ref{thm:exact-hamilton}]
\label{thm:spread-hamilton}
For every $\alpha\in(0,1]$, there exist constants
$C\coloneqq C(\alpha)>0$ and $\gamma\coloneqq\gamma(\alpha)>0$
such that the following holds.
Let $G$ be an $n$-vertex $(p,\beta)$-bijumbled graph and let $\rho \coloneqq \rho(n) \in(0,1]$ satisfy
$$
 \delta(G)\geq\alpha pn,\quad
 \beta\leq\gamma pn,\quad \textrm{and} \quad
 \rho pn\geq C(\log n)^2.
$$
Then, the family of Hamilton cycles of $G$ supports a $C\log n/(pn)$-spread measure. Moreover, following an $[n]$-uniform colouring of $G$, the percolation $G_\rho$ a.a.s. admits a rainbow Hamilton cycle.
\end{theorem}

\subsection{Organisation}

In Section~\ref{sec:preliminaries}, we collect various results facilitating our arguments. 
The proofs of the sublinear-surplus results can all be found in
Section~\ref{sec:sublinear-proofs}; perfect matchings, prescribed trees,
and Hamilton cycles are treated in
Sections~\ref{sec:sublinear-matching},
\ref{sec:sublinear-trees}, and~\ref{sec:sublinear-hamilton},
respectively. In Section~\ref{sec:limitations}, we explain why the McDiarmid-type coupling argument 
fails to yield a sublinear palette surplus for clique factors
under the stated bijumbledness hypotheses.

Our proof of Theorem~\ref{thm:spread-transference} appears in
Section~\ref{sec:proof-transference}. 
The applications to exact palettes are developed in
Section~\ref{sec:exact-rbw-proofs}, with the proofs for perfect
matchings, clique factors, prescribed trees, and Hamilton cycles
presented in Sections~\ref{sec:proof-matching},
\ref{sec:proof-spread-clique-factor}, \ref{sec:proof-trees},
and~\ref{sec:proof-hamilton}, respectively.

\section{Preliminaries}
\label{sec:preliminaries}

\subsection{Balanced bipartitions}\label{sec:balanced-partition}

The following result asserts that graphs of even order with sufficiently large minimum degree admit balanced bipartitions preserving any fixed proportion below one half of every vertex degree to the opposite part; its proof is delegated to Appendix~\ref{sec:lem:balanced-bipartition}.

\begin{lemma}[Balanced bipartition]
\label{lem:balanced-bipartition}
Let $\alpha > 0$ and $r\in(0,1/2)$ be fixed, let $n=2m$, and let $G$ be an $n$-vertex graph with minimum degree 
$\delta(G) \geq \alpha pn$, where $p \coloneqq p(n) \in(0,1]$ satisfies $pn = \omega(\log n)$. Then, for all sufficiently large $n$, the graph $G$ admits a bipartition $V(G)=U\mathbin{\dot\cup}W$ such that $|U|=|W|=m$ and
$$
 \deg_H(v)\geq r\deg_G(v)
 \qquad\text{for every }v\in V(G),
$$
where $H\coloneqq G[U,W]$. In particular, $\delta(H)\geq r \alpha pn$.
Moreover, if $G$ is $(p,\beta)$-bijumbled, then $H$ is bipartite $(p,\beta)$-bijumbled.
\end{lemma}

\subsection{Intrinsic lower bound on bijumbledness}
\label{sec:intrinsic-lower}

A minimum-degree condition on a bijumbled graph yields an
intrinsic lower bound on its bijumbledness parameter. The following
consequence is used in our proofs of both spanning-tree results.

\begin{observation}
\label{obs:density-floor}
Let $\alpha>0$, let $0<p\leq1$, and let $G$ be an $n$-vertex
$(p,\beta)$-bijumbled graph satisfying $\delta(G)\geq\alpha pn$.
Then,
$$
 \beta\geq(1-p)\sqrt{\alpha pn}.
$$
In addition, suppose that $b>0$ is fixed and
$\beta\leq bpn/L$, where $L := L(n)>0$ satisfies
$L=n^{o(1)}$ and $(\log n)^2=o(L)$.  Then,
$$
 \frac{pn}{L(\log n)^2}\longrightarrow\infty.
$$
\end{observation}

\begin{proof}
Fix a vertex $v \in V(G)$ and let $D = \deg_G(v)>0$.
Applying bijumbledness to $X=\{v\}$ and
$Y=\Gamma_G(\{v\})$, for which $|Y|=\ee_G(X,Y)=D$,
we obtain
$$
 (1-p)D\leq\beta\sqrt D.
$$
Division by $\sqrt D$ and the minimum-degree assumption then imply
$\beta\geq(1-p)\sqrt D\geq(1-p)\sqrt{\alpha pn}$.

For the additional assertion, suppose first that $p\leq1/2$.
Combining the lower bound just proved with $\beta\leq bpn/L$
and squaring both sides of the inequality, we obtain
$$
 \frac12\sqrt{\alpha pn}\leq\frac{bpn}{L}
 \quad \textrm{and thus} \quad
 pn\geq\frac{\alpha L^2}{4b^2}.
$$
Consequently,
$$
 \frac{pn}{L(\log n)^2}
 \geq\frac{\alpha L}{4b^2(\log n)^2}
 \longrightarrow\infty.
$$
If $p>1/2$, then $pn>n/2$. Since, moreover, $L=n^{o(1)}$, it follows that
$$
 \frac{pn}{L(\log n)^2}
 >\frac{n}{2L(\log n)^2}
 \longrightarrow\infty.
$$
Both deterministic lower bounds tend to infinity, thus also
handling sequences along which $p$ crosses $1/2$.
\end{proof}

\subsection{McDiarmid palette coupling}\label{sec:coupling}

The aforementioned key result of~\cite[Lemma~2.1]{AignerHorevHefetzPersonTrushkin2026} reads as follows.

\begin{lemma}[McDiarmid-type palette coupling]
\label{lem:coupling}
Let $F$ be a fixed graph and let
$\cF\subseteq\binom{E(F)}h$ be a family of $h$-edge sets.  Colour
the edges of $F$ uniformly from a palette $\cK$ of size $Q$.
Let $S\subseteq\cK$ be of size $\ell\geq h$, and set
\begin{equation}
 \rho = \frac{\ell-h+1}{Q}.
 \label{eq:coupling-rate}
\end{equation}
Then,
\begin{align}
\PP\!\left[
  \text{some $A\in\cF$ is rainbow with all colours in $S$}
 \right]
 \geq
 \PP\!\left[F_\rho\text{ contains some $A\in\cF$}\right].
 \label{eq:coupling-conclusion}
\end{align}
\end{lemma}

\subsection{Percolated expansion}

The next lemma asserts that, at sufficiently large retention rates, percolations of balanced bijumbled bipartite graphs a.a.s. satisfy Hall's condition; its proof is delegated to Appendix~\ref{sec:lem:percolated-Hall}. The adjustable constant in the retention condition preserves the leading palette-surplus constant required in Section~\ref{sec:sublinear-matching}.

\begin{lemma}[Percolated Hall condition for bijumbled graphs]
\label{lem:percolated-Hall}
For every $a>0$ and every $A>2/a$, there exists a constant
$\eta\coloneqq\eta(a,A)>0$ such that the following holds.
Let $H=(U,W;E)$ be a bipartite $(p,\beta)$-bijumbled graph with
$|U|=|W|=m$, where $m\to\infty$, and let
$\vartheta\coloneqq\vartheta(m)\in(0,1]$ satisfy
$$
 \delta(H)\geq apm,
 \qquad
 \beta\leq\eta pm,
 \qquad \textrm{and} \qquad
 -pm\log(1-\vartheta)\geq A\log m,
$$
where $-\log(1-\vartheta)=+\infty$ when $\vartheta=1$.
Then, $H_\vartheta$ a.a.s. admits a perfect matching, uniformly over all such graphs and retention rates.
In particular, the conclusion holds whenever
$\vartheta pm\geq A\log m$.
\end{lemma}

The following lemma supplies the expansion required in both palette
regimes, allowing the expansion factor to depend on $n$; its proof is
delegated to Appendix~\ref{app:tree-expansion-percolation}.

\begin{lemma}[Percolated expansion]
\label{lem:percolated-expansion}
\label{lem:tree-percolation}
\label{lem:sublinear-surplus-expansion-percolation}
Fix $\alpha>0$, let $R := R(n) \geq 2$ be an integer satisfying
$\log R = o(\log n)$, and set
$$
 \kappa = \min\left\{\frac\alpha2,\frac18\right\}, \quad
 \gamma_0 = \min\left\{
  \frac{\alpha}{16},\frac{\sqrt\alpha}{16},\frac14
 \right\},
 \quad \textrm{and} \quad
 \gamma_R = \frac{\gamma_0}{R}.
$$
Let $0 < p := p(n) \leq1$, and let $G$ be an $n$-vertex
$(p,\beta)$-bijumbled graph satisfying
$$
 \delta(G)\geq\alpha pn
 \quad\text{and}\quad
 \beta\leq\gamma_Rpn.
$$
If $\vartheta=\vartheta(n)\in(0,1)$ satisfies
\begin{equation}
 \label{eq:sublinear-expansion-retention-condition}
 -\kappa pn\log(1-\vartheta)\geq(R+3)\log n,
\end{equation}
then a.a.s. $G_\vartheta$ is an $R$-expander.
\end{lemma}

Since constant expansion suffices for Hamiltonicity by  Theorem~\ref{thm:expander-hamilton}, it is convenient to state and use the following immediate consequence of Lemma~\ref{lem:percolated-expansion}.

\begin{corollary}[Percolated constant expansion]
\label{lem:hamilton-percolation}
For every $\alpha>0$ and every fixed integer $R\geq2$,
there exists a constant $\eta\coloneqq\eta(\alpha,R)>0$
such that the following holds, with $\kappa$ as in
Lemma~\ref{lem:percolated-expansion}.
Let $G$ be an $n$-vertex $(p,\beta)$-bijumbled graph satisfying
$$
 \delta(G)\geq\alpha pn
 \quad\text{and}\quad
 \beta\leq\eta pn.
$$
If $\vartheta := \vartheta(n)\in(0,1)$ satisfies
$$
 -\kappa pn\log(1-\vartheta)\geq(R+3)\log n,
$$
then $G_\vartheta$ is a.a.s. an $R$-expander.
\end{corollary}


\section{Rainbow embeddings with sublinear palette surplus}
\label{sec:sublinear-proofs}

In this section, we prove Theorems~\ref{thm:sublinear-surplus-perfect-matching}, ~\ref{thm:sublinear-surplus-tree}, and~\ref{thm:sublinear-surplus-hamilton}; all proofs employ the McDiarmid-type palette coupling lemma, namely Lemma~\ref{lem:coupling}, established in~\cite{AignerHorevHefetzPersonTrushkin2026}. 

\subsection{Perfect matchings: proof of Theorem~\ref{thm:sublinear-surplus-perfect-matching}} \label{sec:sublinear-matching}

The argument consists of three stages. First, we pass to a balanced bipartite
spanning subgraph of the host graph whose minimum degree remains of order $pn$. Second, the palette coupling lemma, namely Lemma~\ref{lem:coupling}, is applied to the resulting bipartite graph, thereby reducing the rainbow problem to the existence of a perfect matching in an independently percolated
subgraph. Third, the probability of a Hall obstruction arising in the percolated bipartite graph is controlled by Lemma~\ref{lem:percolated-Hall}.

\medskip
\noindent\textbf{Step 1. A balanced bipartite spanning subgraph.}\nopagebreak[4]
Set $m = n/2$ as well as
$$
 K = \frac2\alpha + \eta,
 \qquad
 \lambda = \frac12 \left(\alpha+\frac2K\right),
 \qquad \textrm{and} \qquad
 r = \frac{\lambda}{2\alpha}.
$$
Since $K>2/\alpha$, it follows that $\lambda K>2$ and $0<r<1/2$. Let $\eta_{\rm H} := \eta_{\rm H}(\lambda,K) > 0$ be the constant obtained from Lemma~\ref{lem:percolated-Hall} and set $c = \eta_{\rm H}/2$.

Applying Lemma~\ref{lem:balanced-bipartition} to $G$ with the above value of $r$ yields a bipartition
$$
 V(G)=U\mathbin{\dot\cup}W
 \qquad \textrm{such that} \qquad |U| = |W| = n/2,
$$
satisfying $\deg_H(v)\geq r\deg_G(v)$ for every $v\in V(G)$, where $H :=  G[U,W]$. Lemma~\ref{lem:balanced-bipartition} also ensures that $H$ is bipartite $(p,\beta)$-bijumbled. Moreover, setting $D = pm$ and using the equality $2 r \alpha = \lambda$, we obtain
$$
 \delta(H) \geq r\delta(G)\geq r\alpha pn=\lambda D,
 \qquad
 \beta\leq cpn=\eta_{\rm H}D.
$$


\medskip
\noindent\textbf{Step 2. Palette coupling.}\nopagebreak[4]
Let 
$$
 s_{\rm pm} =
 \max\left\{0,
 \left\lceil
 (m-1)\exp\left(\frac{K\log m}{D}\right)-m
 \right\rceil\right\}
$$
denote the \emph{sufficient integer surplus}.
Let $\cK$ be a palette of size $Q \geq m + s_{\rm pm}$, and let $\cF$ be
the family of perfect matchings of $H$.
Lemma~\ref{lem:coupling}, applied with
$$
 F=H, \qquad S=\cK, \qquad h=m, \qquad \textrm{and} \qquad \ell=Q,
$$
asserts that
\begin{equation} \label{eq::McDiarmidMatching}
\PP[H\text{ admits a rainbow perfect matching}]
 \geq \PP[H_\rho\text{ admits a perfect matching}],    
\end{equation}
where the retention rate in~\eqref{eq:coupling-rate} is
$$
 \rho := \frac{Q-m+1}{Q} \geq \frac{s_{\rm pm}+1}{m+s_{\rm pm}},
$$
where the above inequality holds since $1-(m-1)/Q$ is an increasing function of $Q$. It remains to prove that $H_\rho$ a.a.s. admits a perfect matching  and that $s_{\rm pm}$ has the asserted asymptotic value.

\medskip
\noindent\textbf{Step 3. The retention condition and Hall's Theorem.}\nopagebreak[4]
We start by verifying that all conditions of Lemma~\ref{lem:percolated-Hall} are satisfied. Its minimum degree and discrepancy conditions were verified in Step 1. Setting $L_\rho = -D \log(1-\rho)$, its remaining hypothesis is
\begin{equation}
 L_\rho \geq K \log m.
 \label{eq:sublinear-matching-retained-degree}
\end{equation}
For an arbitrary integer $u\geq0$, set
$$
 \rho_u = \frac{u+1}{m+u}
 \quad\textrm{so that}\quad
 1-\rho_u=\frac{m-1}{m+u}.
$$
Since $D>0$ and the exponential function is increasing, the equivalences
\begin{align*}
 -D\log(1-\rho_u)\geq K\log m
 &\quad\Longleftrightarrow\quad
 D\log\left(\frac{m+u}{m-1}\right)\geq K\log m\\*
 &\quad\Longleftrightarrow\quad
 \frac{m+u}{m-1}\geq\exp\left(\frac{K\log m}{D}\right)\\*
 &\quad\Longleftrightarrow\quad
 u\geq(m-1)\exp\left(\frac{K\log m}{D}\right)-m
\end{align*}
hold for $m\geq2$.
By definition, $s_{\rm pm}$ is the least non-negative integer satisfying the last inequality.
Since $\rho\geq\rho_{s_{\rm pm}}$ and $-\log(1-\rho)$ is increasing in $\rho$, condition~\eqref{eq:sublinear-matching-retained-degree} is satisfied.
Lemma~\ref{lem:percolated-Hall} thus ensures that
\begin{equation} \label{eq::aasMatching}
 \PP[H_\rho\text{ admits a perfect matching}] = 1 - o(1).
\end{equation}
It then follows by~\eqref{eq::McDiarmidMatching} and~\eqref{eq::aasMatching} that 
\begin{align*}
\PP[G \textrm{ admits a rainbow perfect matching}] &\geq \PP[H \textrm{ admits a rainbow perfect matching}] \\
&\geq \PP[H_\rho\textrm{ admits a perfect matching}] = 1 - o(1).
\end{align*}

It remains to prove that 
$$
 s_{\rm pm}
 =\left(\frac2\alpha+\eta+o(1)\right)\frac{\log n}{p}.
$$
Let
$$
 x = \frac{K\log(n/2)}{pn/2}=o(1),
$$
where the last equality holds by our assumption that $pn = \omega(\log n)$.
Note that
\begin{align*}
 s_{\rm pm}
 &=(m-1)\exp(x)-m+O(1)
 =m\left(\exp(x)-1\right)-\exp(x)+O(1)\\
 &=mx+O(mx^2)-\exp(x)+O(1)
 =(1+o(1))mx \\
 &=(1+o(1))\frac{K \log m}{p} 
 = \left(\frac2\alpha+\eta+o(1)\right)\frac{\log n}{p},
\end{align*}
where the third equality holds since $\exp(x)-1=x+O(x^2)$, and the fourth equality holds since $x = o(1)$ and $mx \to \infty$.
\hfill $\square$

\subsection{Prescribed rainbow trees: proof of Theorem~\ref{thm:sublinear-surplus-tree}}
\label{sec:sublinear-trees}

The argument consists of three stages. First, we choose the expansion
parameter and apply the palette coupling lemma to the prescribed tree.
Second, Lemma~\ref{lem:percolated-expansion} and
Theorem~\ref{thm:han-yang} establish the required uncoloured
containment probability. Third, the intrinsic density estimate
determines the order of the palette surplus.


\medskip
\noindent\textbf{Step 1. Parameters and palette coupling.} Given $\alpha$ and $\Delta$ per the theorem, fix $T\in\cT(n,\Delta)$. Let
$$
    \kappa =
    \min\left\{\frac\alpha2,\frac18\right\}
$$
be the constant defined in
Lemma~\ref{lem:percolated-expansion}, and set
$$
    c =
    \min\left\{
        \frac\alpha8,
        \frac{\sqrt\alpha}{8},
        \frac12
    \right\}.
$$
Set
$$
    L = L_{n,\Delta},
    \qquad
    R = \lceil L\rceil,
    \qquad \textrm{and} \qquad
    x_n = \frac{(R+3)\log n}{\kappa pn}.
$$
Since $\Delta\geq2$ is fixed, it follows that
$$
    L
    =
    \Delta^{5\sqrt{\log n}}
    =
    \exp\left(O_\Delta(\sqrt{\log n})\right)
    =
    n^{o(1)}.
$$
Moreover, $L\to\infty$, and thus $\log R = o(\log n)$, implying that 
$R$ is an admissible expansion parameter for Lemma~\ref{lem:percolated-expansion}. Additionally, note that the condition on $\beta$ imposed by Lemma~\ref{lem:percolated-expansion} is met. Indeed, our choice of $c$, the inequality $R \leq 2L$, and the assumption $\beta \leq cpn/(4L)$ made in the statement of Theorem~\ref{thm:sublinear-surplus-tree} imply that  
\begin{align*}
    \frac{\beta}{p n}
    \leq
    \frac{c}{2R}
    = \frac{1}{R} \min\left\{
        \frac\alpha{16},
        \frac{\sqrt\alpha}{16},
        \frac14
    \right\} = \gamma_R,
\end{align*}
where $\gamma_R$ is as in Lemma~\ref{lem:percolated-expansion}.


\medskip

Let
$$
    s_T =
    \max\left\{
        0,
        \left\lceil
            (n-2)\exp(x_n)-(n-1)
        \right\rceil
    \right\}
$$
denote the \emph{sufficient integer surplus}; we prove that every integer palette size $Q \geq n-1 + s_T$ suffices. Moreover, at the end of the proof we verify that
$$
    s_T
    =
    \left(\frac1\kappa+o(1)\right) \frac{L_{n,\Delta}\log n}{p}
$$
holds. Let $\cK$ be a palette of size $Q\geq n-1+s_T$, and let
$$
    \cF_T
    :=
    \left\{
        E(T'):\ T'\subseteq G\text{ and }T'\cong T
    \right\}
    \subseteq
    \binom{E(G)}{n-1}.
$$
Applying Lemma~\ref{lem:coupling} with
$$
    F = G,
    \qquad
    S = \cK,
    \qquad
    h = n-1,
    \qquad \textrm{and} \qquad
    \ell = Q
$$
yields
$$
    \PP\left[
        G\text{ admits a rainbow copy of }T
    \right]
    \geq
    \PP\left[
        G_\rho\text{ admits a copy of }T
    \right],
$$
where
$$
    \rho
    :=
    \frac{Q-(n-1)+1}{Q}
    =
    \frac{Q-n+2}{Q}.
$$

\medskip
\noindent\textbf{Step 2. Percolated expansion and tree containment.} It is immediate from the definition of $s_T$ that
$$
    n-1+s_T
    \geq
    (n-2)\exp(x_n).
$$
Since $Q\geq n-1+s_T$, it follows by the definition of $\rho$ that
\begin{align*}
    1-\rho
    =
    \frac{n-2}{Q}
    \leq
    \frac{n-2}{n-1+s_T}
    \leq
    \exp(-x_n).
\end{align*}
Taking logarithms and using the definition of $x_n$, we
obtain
$$
    -\kappa pn\log(1-\rho)
    \geq
    \kappa pn x_n
    =
    (R+3)\log n.
$$
All the hypotheses of
Lemma~\ref{lem:percolated-expansion} are therefore
satisfied, and consequently
$$
    \PP\left[
        G_\rho\text{ is an }R\text{-expander}
    \right]
    =
    1-o(1).
$$
Since $R \geq \Delta^{5\sqrt{\log n}}$, it follows by  Theorem~\ref{thm:han-yang} that every $R$-expander is $\cT(n,\Delta)$-universal. Hence,
\begin{align*}
    \PP\left[
        G\text{ admits no rainbow copy of }T
    \right]
    &\leq
    \PP\left[
        G_\rho\text{ admits no copy of }T
    \right]\\
    &\leq
    \PP\left[
        G_\rho\text{ is not an }R\text{-expander}
    \right] = o(1).
\end{align*}

The error term does not depend on the particular prescribed tree. Therefore,
$$
    \sup_{T\in\cT(n,\Delta)}
    \PP\left[
        G\text{ admits no rainbow copy of }T
    \right]
    =
    o(1).
$$


\medskip
\noindent\textbf{Step 3. The order of the palette surplus.} It remains to determine the order of $s_T$. Observation~\ref{obs:density-floor}, applied with $b=c/4$ and $L=L_{n,\Delta}$, implies
$$
    \frac{pn}{L(\log n)^2}\longrightarrow\infty.
$$
Since $R=(1+o(1))L$, it follows from the definition of $x_n$ that
$$
    x_n
    =O_\alpha\left(\frac{L\log n}{pn}\right)
    =o(1).
$$

It is easy to see that $s_T > 0$ holds for sufficiently large $n$; we may thus assume that $s_T =
(n-2)\exp(x_n)-(n-1)+O(1)$. Using $\exp(x_n)-1=x_n+O(x_n^2)$, we thus obtain
\begin{align*}
    s_T
    &=
    (n-2)\exp(x_n)-(n-1)+O(1) =
    (n-2)(\exp(x_n)-1)-1+O(1)\\
    &=
    nx_n+O(nx_n^2)+O(1) =
    (1+o(1))nx_n =
    (1+o(1))
    \frac{(R+3)\log n}{\kappa p}\\
    &=
    \left(\frac1\kappa+o(1)\right)
    \frac{L\log n}{p}
\end{align*}
where in the fourth equality we use $n x_n^2 = o(n x_n)$ which holds since  $x_n = o(1)$, and the last equality holds since $R=(1+o(1))L$. \hfill $\square$

\subsection{Hamilton cycles: proof of Theorem~\ref{thm:sublinear-surplus-hamilton}}
\label{sec:sublinear-hamilton}

The argument consists of three stages. First, we fix a constant expansion
factor and use the palette coupling lemma to pass to an uncoloured
percolation. Second, Corollary~\ref{lem:hamilton-percolation} and
Theorem~\ref{thm:expander-hamilton} ensure Hamiltonicity of that
percolation. Third, we estimate the resulting palette surplus.

\medskip
\noindent\textbf{Step 1. Parameters and palette coupling.} Fix an integer $R_0 \geq 2$, sufficiently large for the assertion of 
Theorem~\ref{thm:expander-hamilton} to hold, and let $\alpha>0$. Let $\kappa$ and $\eta$ be the constants appearing in Corollary~\ref{lem:hamilton-percolation} when applied with respect to $\alpha$ and $R_0$. Set 

$$
    c = \eta
    \qquad\text{and}\qquad
    A = \frac{R_0+3}{\kappa}.
$$
Since $R_0$ is an absolute constant, both $A$ and $c$ depend only on $\alpha$. Set 
\begin{equation}
y_n = \frac{A\log n}{pn} \qquad \textrm{and} \qquad
    s_H := s_H(n,p) =
    \max\left\{
        0,
        \left\lceil
            (n-1)\exp\left(y_n\right)-n
        \right\rceil
    \right\}.
    \label{eq:sublinear-hamilton-surplus}
\end{equation}
We prove that every integer palette size $Q \geq n + s_H$ suffices.  At
the end of the proof we show that
$$
    s_H =
    (A+o(1))\frac{\log n}{p}.
$$

Let $\cK$ be a palette of size $Q\geq n+s_H$, and let
$$
    \cF_H =
    \left\{
        E(H):\ H\subseteq G\text{ is a Hamilton cycle}
    \right\}
    \subseteq
    \binom{E(G)}n.
$$
Applying Lemma~\ref{lem:coupling} with
$$
    F=G,
    \qquad
    S=\cK,
    \qquad
    h=n,
    \qquad \textrm{and} \qquad
    \ell = Q
$$
yields
$$
    \PP\left[
        G\text{ admits a rainbow Hamilton cycle}
    \right]
    \geq
    \PP\left[
        G_\rho\text{ admits a Hamilton cycle}
    \right],
$$
where
$$
    \rho
    :=
    \frac{Q-n+1}{Q}.
$$

\medskip
\noindent\textbf{Step 2. Percolated expansion and Hamiltonicity.} It is immediate from the definition of $s_H$ that
$$
    n + s_H \geq (n-1) \exp\left(y_n\right).
$$
Since $Q \geq n+s_H$, it follows by the definition of $\rho$ that
\begin{align*}
    1-\rho
    =
    \frac{n-1}{Q}
    \leq
    \frac{n-1}{n+s_H}
    \leq
    \exp\left(-y_n\right).
\end{align*}
Taking logarithms and using the definitions of $y_n$ and
$A$, we obtain
\begin{align*}
    -\kappa pn\log(1-\rho)
    \geq
    \kappa pn y_n
    =
    \kappa A\log n
    =
    (R_0+3)\log n.
\end{align*}
Moreover, $\beta \leq c p n = \eta p n$ holds by the premise of the theorem and our choice of $c$. Hence, all the hypotheses of Corollary~\ref{lem:hamilton-percolation} hold, and consequently
$$
    \PP\left[
        G_\rho\text{ is an }R_0\text{-expander}
    \right]
    =
    1-o(1).
$$
It follows by the choice of $R_0$ and Theorem~\ref{thm:expander-hamilton} that every $R_0$-expander is Hamiltonian. It then follows by Lemma~\ref{lem:coupling} that
\begin{align*}
    \PP\left[
        G\text{ admits a rainbow Hamilton cycle}
    \right]
    &\geq \PP\left[
        G_{\rho}\text{ admits a Hamilton cycle}
    \right] \\
    &\geq
    \PP\left[
        G_\rho\text{ is an }R_0\text{-expander}
    \right]
    =
    1-o(1).
\end{align*}

\medskip
\noindent\textbf{Step 3. The order of the palette surplus.} It remains to estimate $s_H$.  The assumption $pn=\omega(\log n)$
implies
$$
    y_n
    =
    \frac{A \log n}{pn}
    =
    o(1).
$$

It is easy to see that $s_H > 0$ holds for sufficiently large $n$; we may thus assume that $s_H = (n-1)\exp\left(y_n\right)-n+O(1)$. Using $\exp\left(y_n\right)-1=y_n+O(y_n^2)$, we obtain
\begin{align*}
    s_H
    &=
    (n-1)\exp\left(y_n\right)-n+O(1) =
    (n-1)\bigl(\exp\left(y_n\right)-1\bigr)-1+O(1)\\
    &=
    ny_n+O(ny_n^2)+O(1) =
    (1+o(1))ny_n =
    (A+o(1))\frac{\log n}{p},
\end{align*}
where in the fourth equality we use $n y_n^2 = o(n y_n)$ which holds since  $y_n = o(1)$. \hfill $\square$

\subsection{Clique factors: limitations of the McDiarmid-type coupling approach} \label{sec:limitations}

In order to prove a sublinear-surplus result for clique factors in the spirit of Theorems~\ref{thm:sublinear-surplus-perfect-matching}, \ref{thm:sublinear-surplus-tree}, and~\ref{thm:sublinear-surplus-hamilton}, we require Lemma~\ref{lem:coupling} and a result ensuring the a.a.s. appearance of clique factors in percolated bijumbled graphs. The most natural (and perhaps only) result of this nature is Theorem~\ref{thm:morris} which we complement with Lemmas~\ref{lem:degree-regularisation} and~\ref{lem:jumbled-sparsification} in order to handle the percolation. In this section we explain the limitations of this approach.

Fix an integer $k\ge3$ and $\alpha\in(0,1]$, and suppose that
$k\mid n$. A $K_k$-factor has
\[
h=\frac{k-1}{2}n
\]
edges. For a palette of size $Q=h+s$, where $s \ge 0$ is an integer,
Lemma~\ref{lem:coupling}, applied with the entire palette, yields
$$
\mathbb{P}[G\textrm{ admits a rainbow }K_k\textrm{-factor}]
\ge
\mathbb{P}[G_\rho\textrm{ admits a }K_k\textrm{-factor}],
$$
where $\rho = (s+1)/(h+s)$. Since $h=\Theta(n)$, a sublinear surplus is equivalent to $\rho=o(1)$. Hence, it would suffice to prove that a vanishing-rate percolation of the host a.a.s.\ admits an uncoloured $K_k$-factor. Suppose that $G$ is an $n$-vertex $(p,\beta)$-bijumbled graph, satisfying
\[
0 < p \leq 1, \qquad
\delta(G) \ge \alpha pn,
\qquad \textrm{and} \qquad
\beta\le cp^{k-1}n,
\]
where $0 < c \leq \alpha$ is a sufficiently small constant. Set $q = \rho p$, and suppose 
that $qn \gg \log n$. Let
\[
c_1^\circ := c_1^\circ(\alpha)
\qquad \textrm{and} \qquad
c_2^\circ := c_2^\circ(\alpha,\alpha)
\]
be the constants provided by Lemma~\ref{lem:degree-regularisation}. Since
\[
\beta \le cp^{k-1} n \le \alpha pn
\]
and $pn \ge qn \gg \log n$, Lemma~\ref{lem:degree-regularisation} may be applied to yield a spanning $(p,5\beta)$-bijumbled subgraph
$J \subseteq G$ satisfying
\[
c_1^\circ pn\le\delta(J)\le\Delta(J)\le c_2^\circ pn.
\]

Let $C_2$ be the constant provided by Lemma~\ref{lem:jumbled-sparsification}, applied with $K_0=2$, and set
\[
\zeta = c_1^\circ/2
\qquad \textrm{and} \qquad
\eta := \eta(k, \zeta) > 0,
\]
where $\eta$ is the constant appearing in Theorem~\ref{thm:morris}.
By Lemma~\ref{lem:jumbled-sparsification}, with probability at least $1-(2n)^{-2}$, the graph $J_\rho$ is $(q,\widehat\beta)$-bijumbled, where 
\[
\widehat\beta
=
5\rho\beta+
C_2\left(\sqrt{c_2^\circ qn}+\sqrt{\log n}\right).
\]
Moreover, a standard application of Chernoff's bound and a union bound yield
\[
\mathbb{P}\!\left[\delta(J_\rho)<\zeta qn\right]
\le
n\exp\left(-\frac{c_1^\circ qn}{8}\right).
\]
Consequently, both properties hold simultaneously with
probability at least
\[
1-(2n)^{-2}
-n\exp\left(-\frac{c_1^\circ qn}{8}\right)
=1-o(1).
\]

To apply Theorem~\ref{thm:morris} through this discrepancy estimate, we need
\[
\widehat\beta\le\eta q^{k-1}n.
\]
Setting
\[
r_n = \frac{\beta}{p^{k-1}n},
\]
the relevant ratio is
\begin{equation} \label{eq::rn}
\frac{\widehat\beta}{q^{k-1}n}
=
\frac{5r_n}{\rho^{k-2}}
+
\frac{C_2}{\sqrt{(\rho p)^{2k-3}n}}
\left(
\sqrt{c_2^\circ}
+
\sqrt{\frac{\log n}{\rho pn}}
\right).
\end{equation}

If we only require $r_n \leq c$, for some constant $c > 0$, then in order to obtain $\widehat\beta\le\eta q^{k-1}n$, using just the first term in~\eqref{eq::rn} shows that
\begin{equation}
\rho^{k-2} \ge \frac{5c}{\eta}
\label{eq:clique-coupling-inherited}
\end{equation}
must hold. This condition is incompatible with the desired $\rho=o(1)$.
The reason behind this failure is that thinning scales the inherited discrepancy by $\rho$, whereas the discrepancy permitted by Theorem~\ref{thm:morris} scales by $\rho^{k-1}$. Hence, decreasing $c$ but keeping it a constant lowers the bound in
\eqref{eq:clique-coupling-inherited}, but does not make it vanish.

Even if the inherited discrepancy is negligible, absorbing
the term $C_2\sqrt{c_2^\circ qn}$ requires
\begin{equation}
(\rho p)^{2k-3}n
\ge
\frac{C_2^2c_2^\circ}{\eta^2}.
\label{eq:clique-coupling-density}
\end{equation}
In particular, when $p = \Theta \left(n^{-1/(2k-3)} \right)$, this condition also
prevents $\rho=o(1)$. Conversely, a sufficiently large constant
lower bound on $(\rho p)^{2k-3}n$ controls the fluctuation term.
Such a bound also entails
\[
qn = \Omega \left(n^{(2k-4)/(2k-3)} \right) \gg \log n.
\]
Hence, the $\sqrt{\log n}$ term in $\widehat{\beta}$ and the minimum degree
concentration cause no additional difficulty.

There is also an intrinsic reason for the density scale in
\eqref{eq:clique-coupling-density}. By Observation~2.2, any
$(q,\beta_H)$-bijumbled graph $H$ with minimum degree
$\delta(H)\ge\zeta qn$ satisfies
\[
\beta_H\ge(1-q)\sqrt{\zeta qn}.
\]
Consequently, $\beta_H\le\eta q^{k-1}n$ implies
\[
q^{2k-3}n
\ge
\frac{\zeta(1-q)^2}{\eta^2}.
\]
Therefore, when $q=o(1)$, a positive constant lower bound on
$q^{2k-3}n$ is necessary for the simultaneous hypotheses of
Theorem~\ref{thm:morris}, independently of the particular concentration
estimate used above.


For comparison, the discrepancy scales used for perfect matchings and
Hamilton cycles are linear in $q$; their inherited-discrepancy
ratios do not incur the factor $\rho^{-(k-2)}$.
A sharper clique-factor robustness argument may instead exploit
the clique hypergraph. Such an argument must account for the
fact that a clique survives with probability
$\rho^{\binom{k}{2}}$, while cliques sharing graph edges have
dependent survival events.




Finally, the exact-palette argument in Section~\ref{sec:proof-spread-clique-factor} uses a separate auxiliary retention rate to construct a spread measure. It compensates for the same discrepancy comparison by imposing
a stronger host bound, and then applies spread transference.
Thus it neither contradicts nor removes the limitations of the
direct discrepancy verification discussed here.

\section{Exact-palette transference principle: Proof of Theorem~\ref{thm:spread-transference}}
\label{sec:proof-transference}

We prove Theorem~\ref{thm:spread-transference} in two steps.
Lemma~\ref{lem:robust-gives-spread} constructs a spread measure from
uncoloured containment in a random subset; Lemma~\ref{lem:spread-exact-palette}
then applies the rainbow threshold theorem of Han--Yuan~\cite[Theorem~3]{HanYuan2025} to this measure and identifies a sufficient final retention rate.

\begin{lemma}[Spread through robust containment]
\label{lem:robust-gives-spread}
Let $\Omega$ be finite and let $\cA\subseteq\binom{\Omega}{h}$, where
$h\geq1$. Form $R\subseteq\Omega$ by retaining every element
independently with probability $\vartheta\in(0,1]$.  
Then, $\cA$ supports a $(\vartheta/a)$-spread measure whenever
$$
 a:=\PP[\text{$R$ contains a member of $\cA$}]>0.
$$
\end{lemma}

\begin{proof}
Fix a total order on $\cA$. Conditioning on the event
$$
 \mathcal E:=\{\text{$R$ contains a member of $\cA$}\},
$$
let $F(R)$ denote the first member of $\cA$ contained in $R$, and let 
$\mu$ denote the conditional law of $F(R)$ given $\mathcal E$. For every
non-empty $S\subseteq\Omega$, set $s=|S|$. Then 
\begin{align*}
 \mu(\langle S\rangle) 
 & = \sum_{\substack{A \in \cA \\ S \subseteq A}} \mu(\{A\}) 
 = \sum_{\substack{A \in \cA \\ S \subseteq A}} \PP[ F(R) =A \mid\mathcal E] = \PP[S\subseteq F(R)\mid\mathcal E] \\
 & \leq \PP[S\subseteq R\mid\mathcal E]
 \leq \frac{\PP[S\subseteq R]}{\PP[\mathcal E]}
 =\frac{\vartheta^s}{a}
 \leq\left(\frac{\vartheta}{a}\right)^s,
\end{align*}
where the first inequality above is supported by the implication
$S\subseteq F(R)\Rightarrow S\subseteq R$, and the last inequality holds since $a \in (0,1]$ and $s \geq 1$. For $S=\varnothing$, both sides of the spread inequality equal 1. Hence, $\mu$ is the required spread measure. 
\end{proof}

\begin{remark}
Throughout all of our applications, the probability $a$ defined in Lemma~\ref{lem:robust-gives-spread} satisfies $a = 1-o(1)$; for such an $a$, conditioning has a negligible cost and a sampling rate of $\vartheta$ produces a $(1+o(1))\vartheta$-spread measure. 
\end{remark}

The next lemma applies the Han--Yuan~\cite[Theorem~3]{HanYuan2025} theorem (see Theorem~\ref{thm:han-yuan}) to a spread measure supported on  
uncoloured configurations, with an exact palette and independent final
retention. 

\begin{lemma}
\label{lem:spread-exact-palette}
There exists an absolute constant $C_{\rm tr}>0$ such that the following
holds. Let $h\to\infty$ through a sequence of integers; for each
$h$, let $\Omega$ be a finite set and let
$\cA\subseteq\binom{\Omega}{h}$ support a $q$-spread measure,
where $0<q\leq1/2$. Upon 
colouring the elements of $\Omega$ independently and uniformly
from the palette $[h]$,
$$
 \PP\bigl[ \Omega_\rho \;
   \text{contains a rainbow member of }\; \cA
 \bigr]
 = 1 - o_{h}(1)
$$
holds, whenever $\rho\in(0,1]$ satisfies
\begin{equation}\label{eq:q}
 2C_{\rm tr}\cdot q\log h\leq\rho.
\end{equation}
\end{lemma}

\begin{remark} \label{rem::lemma42}
In Lemma~\ref{lem:spread-exact-palette} the $\rho$-percolation is independent of the colouring, and the probability is over both experiments.
In particular, there exists an absolute constant
$c_{\rm sp}>0$ such that, whenever $q\log h\leq c_{\rm sp}$, taking $\rho=1$ a.a.s. yields a rainbow member of $\cA$ in the full coloured ground set.    
\end{remark}

The aforementioned result of Han and Yuan facilitating our proof is stated next. Recall that an $r$-bounded multihypergraph has edges of size at most $r$; it is $\kappa$-spread when the uniform law on its edge copies is
$\kappa^{-1}$-spread. 

\begin{theorem}[Han--Yuan {~\cite[Theorem~3]{HanYuan2025}}]
\label{thm:han-yuan}
There is an absolute constant $C_{\rm HY}>0$ such that the following
holds.  Let $\cH$ be an $r$-bounded $\kappa$-spread
multihypergraph on the finite ground set $\Omega$, and colour the elements of $\Omega$
independently and uniformly from $[k]$, where $k\geq r$.
Independently of the colouring, retain each element with probability
$\rho$. Then, $\Omega_\rho$ contains a rainbow edge of $\cH$ with
probability $1-o_{r}(1)$ whenever
$\rho\in(0,1]$ satisfies 
$$
 \rho\geq\frac{C_{\rm HY}\log r}{\kappa}.
$$
\end{theorem}


We are now in a position to prove Lemma~\ref{lem:spread-exact-palette}. 

\begin{proof}[Proof of Lemma~\ref{lem:spread-exact-palette}]
Set $C_{\rm tr} = C_{\rm HY}$, where $C_{\rm HY}$ is the constant guaranteed by Theorem~\ref{thm:han-yuan}. Let $\mu$ be a $q$-spread measure supported on $\cA$ (not necessarily uniform over $\cA$). Let
$$
\cA_0 = \{A \in \cA: \mu(\{A\}) >0\}
$$
denote the support of $\mu$; we shall approximate its positive masses by
positive rational numbers.
Let 
$$
\mathcal{S} = \left\{S \subseteq \Omega: S \neq \emptyset\; \text{and $S \subseteq A$ for some $A \in \cA_0$} \right\}
$$
be the family consisting of the non-empty sets contained in
at least one member of $\cA_0$.  For every $S\in\mathcal S$,
$$
 \mu(\langle S\rangle)
 \leq q^{|S|}<(2q)^{|S|}
$$
holds since $q>0$ and $|S|\geq1$. Since $\mathcal S$ is finite and
non-empty, we may set
$$
 \delta_0 =
 \min_{S\in\mathcal S}
 \bigl((2q)^{|S|}-\mu(\langle S\rangle)\bigr)>0.
$$

The vector $(\mu(\{A\}))_{A\in\cA_0}$ can be approximated by a rational
probability vector $(\nu(\{A\}))_{A\in\cA_0}$ with positive coordinates
and satisfying
$$
 \sum_{A\in\cA_0}|\nu(\{A\})-\mu(\{A\})|<\delta_0.
$$
Such a choice of vector $\nu$ exists as rational points are dense in the relative
interior of the finite probability simplex, namely 
$$
\left\{(x_1,\ldots,x_{|\cA_0|}):x_i >0, \sum_{i=1}^{|\cA_0|}x_i =1 \right\}.
$$
For every
$S\in\mathcal S$,
$$
 |\nu(\langle S\rangle)-\mu(\langle S\rangle)|
 \leq\sum_{A\in\cA_0}|\nu(\{A\})-\mu(\{A\})|<\delta_0,
$$
and thus 
$$
\nu(\langle S\rangle)\leq \mu(\langle S\rangle) + \delta_0 \leq (2q)^{|S|}.
$$
The same inequality holds when no member of $\cA_0$ contains $S$, and
equality holds if $S=\varnothing$. It follows that the rational probability vector $\nu$, approximating $\mu$, defines a $2q$-spread measure.


Since the coordinates of $\nu$ are positive rational numbers, there is a positive integer $M$ such that $\nu(\{A_i\}) = m_i/M$ for every $A_i \in \cA_0$, where $m_i$ is a positive integer. Since $\nu$ has total mass 1, it follows that $\sum_i m_i = M$. Define an $h$-uniform  multihypergraph $\cH$ on $\Omega$ by including $m_i$ copies of each
hyperedge $A_i\in\cA_0$. Choosing one of the $M$ edges of $\cH$ uniformly  has law $\nu$; hence, $\cH$ is $(2q)^{-1}$-spread.
Finally, \eqref{eq:q} and our choice of $C_{\rm tr}$ imply that
$$
 \frac{C_{\rm HY}\cdot \log h}{(2q)^{-1}}
 =2C_{\rm tr}\cdot q\log h\leq\rho.
$$

An application of Theorem~\ref{thm:han-yuan}, with $r=k=h$, $\kappa = (2q)^{-1}$, and $\rho$, yields a rainbow
hyperedge in $\Omega_\rho$ with probability $1-o(1)$.  
\end{proof}


We are now in a position to prove Theorem~\ref{thm:spread-transference}. 

\begin{proof}[Proof of Theorem~\ref{thm:spread-transference}]
We implement the two-stage argument described at the start of this
section. For each $n$, applying Lemma~\ref{lem:robust-gives-spread} 
with the ground set $\Omega_n$, family $\cA_n$, percolation rate $\vartheta_n$, as well as the assumption that the containment event has
probability $a_n > 0$, yields a $q_n$-spread measure $\mu_n$ supported on $\cA_n$, where $q_n := \vartheta_n/a_n$. 

This proves the first assertion of
Theorem~\ref{thm:spread-transference}.

\medskip
Proceeding to the second assertion of the theorem, let $C_{\rm tr}$ be the constant whose existence is guaranteed by Lemma~\ref{lem:spread-exact-palette}, and let $\rho_n \in (0,1]$ satisfy 
$$
 \rho_n \geq 2C_{\rm tr}q_n\log h_n.
$$
An application of Lemma~\ref{lem:spread-exact-palette} with
$\Omega=\Omega_n$, $\cA=\cA_n$, $h=h_n$, $q=q_n$,
and $\rho=\rho_n$, then yields
$$
 \PP\bigl[ (\Omega_n)_{\rho_n} \;
   \text{contains a rainbow member of }\; \cA_n
 \bigr]
 = 1 - o(1).
$$

\medskip

The final assertion of the theorem holds by Remark~\ref{rem::lemma42}.
\end{proof}


Our proof of Theorem~\ref{thm:spread-transference} employs arguments from various prior works which we wish to credit here. 
Conditional selection from a sampled subgraph appears
in~\cite[proof of Theorem~4.2 and Lemma~7.5]{PhamSahSawhneySimkin2022}.
Rational approximation followed by repeated-edge representation is
used in~\cite[Section~2]{FrankstonKahnNarayananPark2021} and explicitly
described in~\cite[Section~5.2]{HanYuan2025}; our proof records the
finite approximation argument and its factor-two relaxation.
Accordingly, the present transference theorem, namely Theorem~\ref{thm:spread-transference}, is a quantitative
assembly of these mechanisms.
It separates a host-dependent uncoloured estimate from the general treatment of Han--Yuan of random colours.

\section{Exact-palette rainbow spanning configurations}\label{sec:exact-rbw-proofs}

In this section, we prove Theorems~\ref{thm:spread-matching},
~\ref{thm:spread-clique-factor}, ~\ref{thm:spread-trees},
and~\ref{thm:spread-hamilton}; each is derived from
Theorem~\ref{thm:spread-transference}. Their immediate rainbow conclusions
are precisely Theorems~\ref{thm:exact-matching},
~\ref{thm:exact-clique-factor}, ~\ref{thm:exact-trees},
and~\ref{thm:exact-hamilton}, respectively.

\subsection{Perfect matchings: proof of Theorem~\ref{thm:spread-matching}}
\label{sec:proof-matching}

The proof consists of three stages. First, we establish uncoloured containment at an auxiliary retention rate, then determine the supported spread parameter, and finally verify the retention condition for the rainbow conclusion.

\medskip
\noindent\textbf{Step 1. Auxiliary percolation and uncoloured containment.}\nopagebreak[4]
Given $c>0$ as per Theorem~\ref{thm:spread-matching}, set $c_0 = \min\{c,1\}$ and $a = c_0/2$. Choose $A = 4/a > 2/a$, and let $\eta := \eta(a,A) > 0$ be the constant supplied by Lemma~\ref{lem:percolated-Hall}.
Set $\gamma = \eta/2$ and choose $C := C(c)$ to be sufficiently large so that
$$
 C\geq\max\left\{8A,8C_{\rm tr}A\right\},
$$
where $C_{\rm tr}$ is the absolute constant appearing in Theorem~\ref{thm:spread-transference}.

Let $G$ and $\rho$ satisfy the hypotheses of Theorem~\ref{thm:spread-matching} and set $m = n/2$. The inequalities 
$$
 pn\geq\rho pn\geq C(\log n)^2\geq C\log n
$$
hold whenever $n$ is sufficiently large; in particular, $pn=\omega(\log n)$.
Since, moreover, $\delta(G) \geq c_0 p n$, an application of
Lemma~\ref{lem:balanced-bipartition} with $\alpha = c_0$ and $r = 1/4$ yields a balanced partition $V(G) = U \mathbin{\dot\cup} W$ such that the graph $H := G[U,W]$ satisfies
$$
 \delta(H)\geq c_0pn/4=apm,
 \quad \text{as well as} \quad
 \beta\leq\gamma pn=\eta pm,
$$
and is bipartite $(p,\beta)$-bijumbled.

Set $\vartheta_n = 2 A \log n/(pn)$. Since $pn \geq C \log n$ and $C \geq 8A$, it follows that $0 < \vartheta_n \leq 1/4$. Moreover,
$$
 \vartheta_n pm = A \log n \geq A\log m.
$$
All hypotheses of the final assertion of Lemma~\ref{lem:percolated-Hall} are therefore satisfied, and thus $H_{\vartheta_n}$ admits a perfect matching with probability $1-o(1)$. Observe that every perfect matching of $H_{\vartheta_n}$ is also a perfect matching of $G_{\vartheta_n}$.


\medskip
\noindent\textbf{Step 2. The spread parameter.}\nopagebreak[4]
Let $\Omega_n = E(G)$ and let $\cA_n$ be the family of edge-sets of perfect matchings of $G$. Every member of $\cA_n$ has $h_n := n/2$ elements.
Step 1 establishes that
$$
 a_n := \PP[\text{$G_{\vartheta_n}$ contains a member of $\cA_n$}] = 1-o(1).
$$
In particular, $a_n \geq 1/2$ for all sufficiently large $n$; restricting the sequence to those $n$ ensures $a_n>0$ throughout, as required by Theorem~\ref{thm:spread-transference}.
The spread parameter satisfies
$$
 0<q_n:=\frac{\vartheta_n}{a_n}
 \leq\frac{4A\log n}{pn}\leq\frac12,
$$
where the last inequality holds since $pn\geq C\log n$ and $C\geq8A$.
Since $h_n\to\infty$, the first assertion of Theorem~\ref{thm:spread-transference} provides a $q_n$-spread measure on $\cA_n$.
As $C\geq4A$, this measure is also $C\log n/(pn)$-spread, proving the spread estimate asserted in Theorem~\ref{thm:spread-matching}.

\medskip
\noindent\textbf{Step 3. Rainbow containment after colouring and percolation.}\nopagebreak[4]
It remains to verify the percolation condition for the rainbow conclusion.
Since $q_n \leq 4A\log n/(pn)$, $C \geq 8 C_{\rm tr} A$, and $\rho pn\geq C(\log n)^2$, it follows that
$$
 2C_{\rm tr}q_n\log h_n
 \leq\frac{8C_{\rm tr}A\log n\log(n/2)}{pn}
 \leq\frac{C(\log n)^2}{pn}
 \leq\rho.
$$
The second assertion of Theorem~\ref{thm:spread-transference} therefore applies to an $[n/2]$-uniform colouring of $G$, followed by an independent $\rho$-percolation of $G$.
A rainbow perfect matching emerges in $G_\rho$ a.a.s., with probability taken over both experiments.
\hfill $\square$

\subsection{Clique factors}
\label{sec:proof-spread-clique-factor}

In this section, we prove Theorem~\ref{thm:spread-clique-factor}.
The next two lemmas, taken from our companion paper~\cite{AignerHorevHefetzPersonTrushkin2026}, facilitate our proof.  


\medskip
The first of these lemmas allows one to pass to a spanning subgraph whose minimum and maximum degrees are both of order $pn$, at the cost of a constant factor in the bijumbledness parameter.

\begin{lemma}[Degree regularisation -- Lemma 3.1 in~~\cite{AignerHorevHefetzPersonTrushkin2026}]
\label{lem:degree-regularisation}
Given constants \(c, a > 0\), there exist constants  
\(c_1^{\circ} := c_1^{\circ}(c) > 0\),
\(c_2^{\circ} := c_2^{\circ}(c,a) > 0\), and 
\(d_{\star} := d_{\star}(c) > 0\) with the following property. For every positive integer \(n\), every \(0 < p := p(n) \leq 1\), and every \(\beta := \beta(n) \geq 0\), set \(d = p n\) and suppose that \(d \geq d_{\star}\). Then, every \(n\)-vertex
\((p,\beta)\)-bijumbled graph \(G\) satisfying
\[
 \delta(G) \geq c d \quad \textrm{and} \quad 
 \beta \leq a d,
\]
has a spanning subgraph \(G^{\circ}\) that is \((p,5\beta)\)-bijumbled and such that
\[
 c_1^{\circ}d
 \leq
 \delta(G^{\circ})
 \leq
 \Delta(G^{\circ})
 \leq
 c_2^{\circ}d.
\]
\end{lemma}


The second lemma controls the bijumbledness parameter post
percolation. 

\begin{lemma}[Bijumbledness inheritance post percolation -- Lemma 3.2 in~~\cite{AignerHorevHefetzPersonTrushkin2026}]
\label{lem:jumbled-sparsification}
Let $K_0 > 0$ be a constant and let $F$ be an $n$-vertex
$(p,\beta)$-bijumbled graph. Then, there exists a constant
$C_{K_0} > 0$, depending only on $K_0$, such that for every
$0<\rho<1$,
\begin{align*}
 &\Pr\!\left[
  F_\rho\text{ is }\left(
   \rho p,
   \rho\beta+C_{K_0}\bigl(
     \sqrt{\rho\Delta(F)}+\sqrt{\log n}
   \bigr)
  \right)\text{-bijumbled}
 \right]
 \geq1-(2n)^{-K_0}.
\end{align*}
\end{lemma}


We are now in a position to prove Theorem~\ref{thm:spread-clique-factor}. 

\begin{proof}[Proof of Theorem~\ref{thm:spread-clique-factor}]
The proof consists of three stages. We first choose an auxiliary retention rate for which an uncoloured percolation a.a.s. admits a $K_k$-factor.
We then obtain the spread parameter from the first assertion of
Theorem~\ref{thm:spread-transference}, and finally we verify the retention
condition in its rainbow assertion.

\medskip
\noindent\textbf{Step 1. Auxiliary retention and uncoloured containment.}\nopagebreak[4]
Fix $k\geq3$ and $\alpha\in(0,1]$. Let $c_1^{\circ} := c_1^{\circ}(\alpha)$ and $c_2^{\circ} := c_2^{\circ}(\alpha, \alpha)$ be the constants whose existence is ensured by Lemma~\ref{lem:degree-regularisation}. Let $\eta := \eta(k, c_1^{\circ}/2) > 0$ be the constant guaranteed by Theorem~\ref{thm:morris}. Let $C_{\rm tr}$ and $C_2$
be the constants defined in Theorem~\ref{thm:spread-transference} and
Lemma~\ref{lem:jumbled-sparsification} with $K_0 = 2$, respectively.
Set 
$$
 b = \min\left\{\frac14,\frac{1}{16C_{\rm tr}}\right\} \qquad \textrm{and} \qquad \sigma = 2b.
$$
Set 
$$
 \gamma =  \frac{\eta b^{k-2}}{10}
$$
and choose $C$ sufficiently large so as to ensure
$$
 (bC)^{k-3/2}\geq\frac{4C_2 \sqrt{c_2^{\circ}}}{\eta}.
$$
Note that 
$$
 \beta \leq \gamma \left(\frac{\rho}{\log n}\right)^{k-2} p^{k-1} n \leq \alpha pn, 
$$
where the last inequality holds for sufficiently large $n$. The density assumption $\rho p \geq Cn^{-1/(2k-3)} \log n$ implies that the conditions of Lemma~\ref{lem:degree-regularisation} are met. Hence, there exists a $(p,5\beta)$-bijumbled spanning subgraph $J \subseteq G$ satisfying 
$$
 c_1^{\circ} p n
 \leq
 \delta(J)
 \leq
 \Delta(J)
 \leq
 c_2^{\circ} p n.
$$
Set
$$
 \vartheta = \frac{b\rho}{\log n}
 \qquad \textrm{and} \qquad t = \vartheta p,
$$
and note that 
\begin{equation}
\label{eq:clique-sampling-scale}
 t\geq bC n^{-1/(2k-3)}
 \qquad \textrm{and thus} \qquad tn \geq b C n^{(2k-4)/(2k-3)} \gg \log n.
\end{equation}


Applying Lemma~\ref{lem:jumbled-sparsification} to $J$ ensures that, with probability at least $1 - (2n)^{-2} = 1 - o(1)$,
the percolation $J_\vartheta$ is $(t,\beta')$-bijumbled, where  $\beta' = 5\vartheta\beta + C_2 \bigl(\sqrt{\vartheta \Delta(J)}+\sqrt{\log n}
   \bigr) \leq \eta t^{k-1}n$. To see the validity of the bound on $\beta'$, note first that
\begin{align*}
 5\vartheta\beta
 \leq5\vartheta\gamma
       \left(\frac{\rho}{\log n}\right)^{k-2}p^{k-1}n
 =5\gamma b^{-(k-2)}t^{k-1}n
 \leq\frac\eta2t^{k-1}n.
\end{align*}
Next, since $tn \gg \log n$, it follows that $C_2 \bigl(\sqrt{\vartheta \Delta(J)} + \sqrt{\log n}\bigr) \leq 2 C_2 \sqrt{c_2^{\circ} t n}$. It thus follows by~\eqref{eq:clique-sampling-scale} and by our choice of $C$ that
$$
 \frac{2C_2 \sqrt{c_2^{\circ} t n}}{t^{k-1}n}
 =\frac{2C_2 \sqrt{c_2^{\circ}}}{\sqrt{t^{2k-3}n}}
 \leq\frac{2C_2 \sqrt{c_2^{\circ}}}{(bC)^{k-3/2}}
 \leq\frac\eta2.
$$

Additionally, since $t n \gg \log n$, a standard application of Chernoff's bound and a union bound shows that $\delta(J_\vartheta) \geq c_1^{\circ} t n/2$ holds with probability at least
$$
1 - n \exp \left(- \frac{c_1^{\circ} tn}{8} \right) = 1 - o(1).
$$

Since, moreover, $k \mid n$ and $t>0$, Theorem~\ref{thm:morris} may be applied to show that a.a.s. $J_\vartheta$ admits a $K_k$-factor, and thus
$$
 a_n:=\PP[\text{$G_\vartheta$ admits a $K_k$-factor}]
 \geq\PP[\text{$J_\vartheta$ admits a $K_k$-factor}]
 =1-o(1).
$$

Let $\cF_k(G)$ be the family of edge-sets of $K_k$-factors of
the original graph $G$; every member of this family has
$$
 h := \frac nk\binom{k}{2}=\frac{k-1}{2}n
$$
edges. 


\medskip
\noindent\textbf{Step 2. Construction of the spread measure.}\nopagebreak[4]
Since $a_n = 1 - o(1)$, there exists some $n_0$ such that $a_n \geq 1/2$ holds for any $n \geq n_0$; for the remainder of the proof assume that $n \geq n_0$. The first assertion of Theorem~\ref{thm:spread-transference}, applied to the ground set
$E(G)$, the family $\cF_k(G)$, and the retention
probability $\vartheta$, ensures the existence of a $q_n$-spread measure on the family of $K_k$-factors of $G$ such that 
$$
 0<q_n:=\frac{\vartheta}{a_n}
 \leq2\vartheta
 =\frac{2b\rho}{\log n} = \frac{\sigma\rho}{\log n} 
 \leq\frac12,
$$
where the last inequality holds for sufficiently large $n$. The first assertion of Theorem~\ref{thm:spread-clique-factor} is thus established. 

\medskip
\noindent\textbf{Step 3. Rainbow containment after colouring and retention.}\nopagebreak[4]
Since $k$ is fixed, $\log h\leq2\log n$ holds for all sufficiently
large $n$. Our choice of $b$ then ensures
$$
 2C_{\rm tr}q_n\log h
 \leq2C_{\rm tr}\frac{2b\rho}{\log n}\,2\log n
 =8C_{\rm tr}b\rho \leq \rho.
$$
All hypotheses of Theorem~\ref{thm:spread-transference} are thus satisfied, establishing the rainbow assertion of Theorem~\ref{thm:spread-clique-factor}.
\end{proof}


\subsection{Prescribed spanning trees: proof of Theorem~\ref{thm:spread-trees}}
\label{sec:proof-trees}

The proof consists of three stages. We first choose an auxiliary retention rate for which an uncoloured percolation a.a.s. contains a copy of any prescribed tree $T\in\cT(n,\Delta)$. We then obtain the spread parameter from the first assertion of Theorem~\ref{thm:spread-transference}, and finally verify the retention condition in its rainbow assertion.

\medskip
\noindent\textbf{Step 1. Auxiliary retention and uncoloured containment.}\nopagebreak[4] Fix $\alpha \in (0,1]$, and let $\gamma_0 \coloneqq \gamma_0(\alpha)$ and $\kappa \coloneqq \kappa(\alpha)$
be the constants whose existence is guaranteed by
Lemma~\ref{lem:percolated-expansion}.
Choose constants $0 < \gamma \leq \gamma_0/2$ and
$$
 C\geq
 \max\left\{
  1,\frac6\kappa,\frac{12C_{\rm tr}}{\kappa}
 \right\},
$$
where $C_{\rm tr}$ is the constant appearing in
Theorem~\ref{thm:spread-transference}.

Fix an integer $\Delta\geq2$, 
let $G$ and $\rho\in(0,1]$ satisfy the hypotheses 
of Theorem~\ref{thm:spread-trees}, and set
$$
 d = p n, \qquad
 L := L_{n,\Delta } = \Delta^{5\sqrt{\log n}},
 \qquad \textrm{and} \qquad
 R = \lceil L\rceil.
$$
Using this notation and assuming $n$ to be sufficiently large,
$$
 \delta(G)\geq\alpha d,\qquad
 \beta\leq\gamma\frac dL \leq\gamma_0\frac{d}{R},\qquad \textrm{and} \qquad
 \rho d\geq C L(\log n)^2
$$
hold by the premise of Theorem~\ref{thm:spread-trees}. Observation~\ref{obs:density-floor}, applied with $b = \gamma$
and $L = L_{n,\Delta}$, ensures that
\begin{equation}
\label{eq:tree-density-floor}
 \frac{d}{L(\log n)^2}\longrightarrow\infty.
\end{equation}
This estimate will establish the $o(1/\log n)$ bound on the spread parameter in~\eqref{eq:q_T}.

Set
$$
 \vartheta = \frac{(R+3)\log n}{\kappa d}.
$$
Since $R+3 \leq 3L$ holds for sufficiently large $n$, it follows from~\eqref{eq:tree-density-floor} that
$$
 0<\vartheta
 \leq\frac3\kappa\frac{L\log n}{d}
 =o\left(\frac1{\log n}\right) < 1,
$$
where the last inequality holds for sufficiently large $n$. Finally, $R\geq2$, $\log R = o(\log n)$, and
$$
 -\kappa d\log(1-\vartheta)
 \geq\kappa d\vartheta
 = (R+3)\log n,
$$
where in the above inequality we used the fact that $-\log(1-u) \geq u$ holds whenever $0<u<1$. All hypotheses of Lemma~\ref{lem:percolated-expansion} are thus met.


Let $\mathcal E_n$ denote the event that $G_\vartheta$ is an
$R$-expander; then, $\PP[\mathcal E_n]=1-o(1)$ holds, by Lemma~\ref{lem:percolated-expansion}. It thus follows by Theorem~\ref{thm:han-yang} that $G_\vartheta$ is a.a.s. $\cT(n,\Delta)$-universal.


\medskip
\noindent\textbf{Step 2. Construction of the spread measure.} For each prescribed $T\in\cT(n,\Delta)$, let
$\cA_T\subseteq\binom{E(G)}{n-1}$ be the family of edge-sets of copies of $T$ in $G$. Set
$$
 a_T = \PP[\text{$G_\vartheta$ contains a copy of $T$}]
$$
and note that
$$
 \mathcal E_n
 \subseteq
 \{\text{$G_\vartheta$ contains a copy of $T$}\}
$$
holds for every such $T$. Consequently,
$$
 \inf_{T\in\cT(n,\Delta)}a_T
 \geq\PP[\mathcal E_n]
 =1-o(1).
$$
In particular, $a_T \geq 1/2$ holds for every prescribed tree $T$ and every  
sufficiently large $n$. We may therefore set the spread parameter to 
$q_T:=\vartheta/a_T$; then, 
$$
 0<q_T
 \leq2\vartheta
 \leq\frac6\kappa\frac{L\log n}{d}
 \leq C\frac{L\log n}{pn}
 =o\left(\frac1{\log n}\right)
$$
holds uniformly over $T$, where the last equality holds by~\eqref{eq:tree-density-floor}.
Hence, $q_T \leq 1/2$ for sufficiently large $n$, uniformly over all prescribed trees.

\medskip

To apply Theorem~\ref{thm:spread-transference}, fix an
arbitrary sequence of prescribed trees
$T=T_n\in\cT(n,\Delta)$ and set
$$
 \Omega_n = E(G),\qquad
 \cA_n = \cA_{T_n},\qquad
 h_n = n-1,\qquad \textrm{and} \qquad
 \vartheta_n = \vartheta.
$$
The event that $(\Omega_n)_{\vartheta_n}$ contains a
member of $\cA_n$ is precisely the event that
$G_\vartheta$ contains a copy of $T_n$; by definition its probability
is $a_{T_n}$. For all sufficiently large $n$, these
probabilities are positive and $\vartheta_n/a_{T_n} \leq 1/2$.
Since $h_n\to\infty$, the first assertion of
Theorem~\ref{thm:spread-transference} provides a
$q_{T_n}$-spread measure on $\cA_{T_n}$.
The sequence of prescribed trees was arbitrary, and
the preceding bounds are independent of its choice;
this then establishes the first assertion of Theorem~\ref{thm:spread-trees}.

\medskip
\noindent\textbf{Step 3. Colouring and the final retention rate.} It remains to verify the rainbow assertion of Theorem~\ref{thm:spread-trees}. It follows by our choice of $C$ and by the assumed inequality $\rho d\geq C L(\log n)^2$ that
$$
 \begin{aligned}
 2C_{\rm tr}q_T\log(n-1)
 &\leq
 \frac{12C_{\rm tr}}{\kappa}
 \frac{L\log n\,\log(n-1)}{d}
 \leq
 \frac{12C_{\rm tr}}{\kappa}
 \frac{L(\log n)^2}{d}
 \leq
 C\frac{L(\log n)^2}{d}
 \leq \rho.
 \end{aligned}
$$
We can thus apply the second assertion of
Theorem~\ref{thm:spread-transference} with
$\rho_n=\rho$.
Colour $G$ $[n-1]$-uniformly and subsequently percolate at rate $\rho$, independently of the colouring.
The resulting coloured graph $G_\rho$ a.a.s. contains a
rainbow member of $\cA_T$.

\medskip

Finally, we verify uniform convergence of the rainbow
failure probabilities.
For each $n$, choose
$T_n\in\cT(n,\Delta)$ maximising
$$
 \PP[\text{$G_\rho$ contains no rainbow copy of $T_n$}].
$$
A maximum exists because there are finitely many
$n$-vertex trees up to isomorphism.
All estimates above apply to this sequence.
Theorem~\ref{thm:spread-transference} therefore implies
that its failure probability tends to zero, and hence
$$
 \sup_{T\in\cT(n,\Delta)}
 \PP[\text{$G_\rho$ contains no rainbow copy of $T$}]
 =o(1).
$$
\hfill $\square$

\begin{remark}
\label{rem:not-simultaneous-trees}
On the expansion event, the auxiliary uncoloured graph
$G_\vartheta$ contains every member of $\cT(n,\Delta)$.
We nevertheless construct the spread measure and apply
Theorem~\ref{thm:spread-transference} separately for each
prescribed tree.
Accordingly, Theorems~\ref{thm:exact-trees}
and~\ref{thm:spread-trees} assert uniform convergence
of the failure probabilities for prescribed trees;
they do not assert that a single coloured percolation
simultaneously contains rainbow copies of every tree
in $\cT(n,\Delta)$.
\end{remark}

\subsection{Hamilton cycles: proof of Theorem~\ref{thm:spread-hamilton}}
\label{sec:proof-hamilton}

The proof consists of three stages. We first choose an auxiliary retention rate for which an uncoloured percolation a.a.s. admits a Hamilton cycle.
We then obtain the spread parameter from the first assertion of
Theorem~\ref{thm:spread-transference}, and finally verify the retention
condition in its rainbow assertion.


\medskip
\noindent\textbf{Step 1. Auxiliary retention and uncoloured containment.} Fix an integer $R \geq 2$ to be sufficiently large for the assertion of 
Theorem~\ref{thm:expander-hamilton} to hold.
Let $\eta := \eta(\alpha,R)$ and $\kappa := \kappa(\alpha)$ be provided by
Corollary~\ref{lem:hamilton-percolation}, and choose $0 < \gamma \leq \eta$.
Choose a sufficiently large constant $C := C(\alpha)$ so as to satisfy 
$$
 C\geq\max\left\{
  1,\frac{8(R+3)}{\kappa},
  \frac{4C_{\rm tr}(R+3)}{\kappa}
 \right\},
$$
where $C_{\rm tr}$ is the constant appearing in the statement of 
Theorem~\ref{thm:spread-transference}.

Let $G$ and $\rho$ be as in the premise of Theorem~\ref{thm:spread-hamilton} and set $d = p n$. Note that
\begin{equation} \label{eq::densityHam}
    d \geq \rho d \geq C(\log n)^2 \geq C \log n
\end{equation}
holds for all sufficiently large $n$. Set the retention probability to be 
$$
 \vartheta = \frac{(R+3)\log n}{\kappa d}.
$$
The above inequalities and our choice of $C$ imply
$$
 0 < \vartheta \leq \frac{R+3}{\kappa C} \leq \frac18.
$$
Furthermore,
$$
 -\kappa d\log(1-\vartheta)
 \geq\kappa d\vartheta
 =(R+3)\log n,
$$
where we used $-\log(1-u)\geq u$ for $0<u<1$.
All hypotheses of Corollary~\ref{lem:hamilton-percolation} are thus satisfied, implying that $G_\vartheta$ is a.a.s. an $R$-expander.
It then follows by Theorem~\ref{thm:expander-hamilton} and the choice of $R$ that
$$
 a_n := \PP[\text{$G_\vartheta$ admits a Hamilton cycle}]
 = 1 - o(1).
$$

\medskip
\noindent\textbf{Step 2. Construction of the spread measure.} Let $\Omega_n = E(G)$ and let $\cA_n \subseteq \binom{E(G)}{n}$ be the family of edge-sets of Hamilton cycles in $G$. Then, $h_n = n \to \infty$, and the containment probability for the $\vartheta$-percolation of $\Omega_n$ is $a_n$.
Since $a_n \geq 1/2$ for all sufficiently large $n$, it follows that
$$
 0<q_n:=\frac{\vartheta}{a_n}
 \leq2\vartheta
 =\frac{2(R+3)}{\kappa}\frac{\log n}{pn}
 \leq C\frac{\log n}{pn} \leq 1/2,
$$
where the last inequality holds for sufficiently large $n$. Theorem~\ref{thm:spread-transference} therefore provides a
$C\log n/(pn)$-spread measure on $\cA_n$.
The first assertion of Theorem~\ref{thm:spread-hamilton} is thus established.


\medskip
\noindent\textbf{Step 3. Colouring and the final retention rate.} It follows by~\eqref{eq::densityHam} and by our choice of $C$ that
$$
 2C_{\rm tr}q_n\log n
 \leq\frac{4C_{\rm tr}(R+3)}{\kappa}
              \frac{(\log n)^2}{pn}
 \leq\frac{C(\log n)^2}{pn}
 \leq\rho.
$$
All hypotheses of Theorem~\ref{thm:spread-transference} are thus satisfied, establishing the rainbow assertion of Theorem~\ref{thm:spread-hamilton}.
\hfill $\square$


\section*{AI disclosure} ChatGPT Plus was used for \LaTeX\  support; in particular, it was given the content for all the tables appearing in the manuscript and it was asked to generate the \LaTeX\  code for those. It was also asked to find potentially awkward phrases and formulations in English and to propose alternatives which, on occasion, were adopted.

\bibliographystyle{amsplain}
\bibliography{ExactPaletteLit}

\appendix

\section{Proofs omitted from Section~\ref{sec:preliminaries}}
\label{app:preliminaries}

\subsection{Proof of Lemma~\ref{lem:balanced-bipartition}} \label{sec:lem:balanced-bipartition}

Draw uniformly at random a bipartition
$V(G)=U\mathbin{\dot\cup}W$ with $|U|=|W|=m$.
Fix $v\in V(G)$ and condition on the side of the bipartition containing
$v$. The opposite side is a uniformly chosen $m$-subset of the other
$n-1$ vertices. Hence, the cross-degree $D_v$ has a hypergeometric
distribution with mean
$$
 \mu_v := \frac{m}{n-1} \deg_G(v) \geq \frac12 \deg_G(v).
$$
Set $t = 1-2r \in (0,1)$; then, $r \deg_G(v) \leq 2 r \mu_v = (1-t) \mu_v$. It thus follows by Chernoff's bound for the hypergeometric distribution that 
$$
 \PP[D_v < r\deg_G(v)]
 \leq\PP[D_v < (1-t) \mu_v]
 \leq\exp\left(-\frac{t^2\mu_v}{2}\right)
 \leq\exp\left(-\frac{t^2 \alpha pn}{4}\right),
$$
where the last inequality holds by the assumed lower bound on $\delta(G)$.
This bound holds regardless of which side contains $v$, and thus also without that conditioning. A union bound over all $n$ vertices bounds the probability of any failed degree inequality by
$$
 n\exp\left(-\frac{t^2 \alpha pn}{4}\right)=o(1),
$$
where the equality holds since $pn=\omega(\log n)$ whereas $t, \alpha > 0$ are fixed. Therefore, an admissible bipartition exists for all sufficiently large $n$. Finally, for every $X\subseteq U$ and $Y\subseteq W$, the identity
$e_H(X,Y)=\ee_G(X,Y)$ transfers the corresponding bijumbledness condition  from $G$ to $H$.
\hfill$\square$

\subsection{Proof of Lemma~\ref{lem:percolated-Hall}}\label{sec:lem:percolated-Hall}

Set
$$
 g = a-\frac2A>0,
 \qquad
 \tau = \min\left\{\frac14,\frac g4\right\},
 \qquad
 \eta = \min\left\{\frac g4,\frac12\sqrt{\frac\tau2}\right\},
 \qquad \textrm{and} \qquad
 b = a-\tau-\eta.
$$
Since $\tau,\eta\leq g/4$, it follows that
\begin{equation}
 b\geq a-\frac g2=\frac2A+\frac g2
 \qquad\text{which in turn implies}\qquad bA>2.
 \label{eq:sublinear-matching-small-cut-constant}
\end{equation}

If Hall's condition fails in $H_\vartheta$, then by potentially interchanging the roles of $U$ and $W$, a standard argument shows that there are sets $S \subseteq U$ and $B \subseteq W$ such that
\begin{equation}
 |S|=t,
 \qquad
 |B|=t-1,
 \qquad \textrm{and} \qquad
 1 \leq t \leq \frac{m+1}{2},
 \label{eq:det-hall-witness}
\end{equation}
and such that $E_{H_\vartheta}(S, Z) = \varnothing$, where $Z := W \setminus B$.


Set $D = p m$. Suppose first that $t\leq\tau m$.
It then follows by the minimum degree and bijumbledness assumptions that
\begin{align*}
 e_H(S,Z)
 &=\sum_{v\in S}\deg_H(v)-e_H(S,B)
 \geq aDt-pt(t-1)-\beta\sqrt{t(t-1)}\\*
 &\geq(a-\tau-\eta)Dt = bDt,
\end{align*}
where the last inequality follows by $p(t-1)\leq p\tau m=\tau D$,
$\sqrt{t(t-1)}\leq t$, and $\beta\leq\eta D$.

Suppose then that $\tau m<t\leq(m+1)/2$, and set
$x = t/m$ and $z = |Z|/m = (m-t+1)/m$.
Then, $x>\tau$ and $z\geq1/2$, and thus $xz \geq \tau/2$.
The choice of $\eta$ then ensures that
$$
 \frac\beta D\leq\eta
 \leq\frac12\sqrt{\frac\tau2}
 \leq\frac12\sqrt{xz}.
$$
Using the bijumbledness of $H$, we obtain
\begin{align*}
 e_H(S,Z)
 &\geq p|S||Z|-\beta\sqrt{|S||Z|}
 =Dm\sqrt{xz}\left(\sqrt{xz}-\frac\beta D\right)
 \geq\frac12 Dm xz
 \geq \frac\tau4Dm.
\end{align*}
In both ranges $e_H(S, Z) > 0$ holds for each such pair and thus $H$ itself does admit a perfect matching. Consequently, we may henceforth assume that $0<\vartheta<1$ and set $L_\vartheta = - D \log(1-\vartheta)$.

For each fixed pair $(S,B)$, the independence of edge retention implies
$$
 \PP[e_{H_\vartheta}(S,Z)=0]
 =(1-\vartheta)^{e_H(S,Z)}.
$$
A union bound over all choices of $S$ and $B$ yields 
\begin{align}
 \PP[H_\vartheta\text{ violates Hall's condition}]
 &\leq
 2\sum_{1\leq t\leq\tau m}
 \binom mt\binom m{t-1}(1-\vartheta)^{bDt}
 +2\cdot4^m(1-\vartheta)^{\tau Dm/4},
 \label{eq:sublinear-matching-hall-union-bound}
\end{align}
where the factor 2 in both terms comes from taking $S \subseteq U$ or $S \subseteq W$.

For the first term in~\eqref{eq:sublinear-matching-hall-union-bound} note that
$$
 2\binom mt\binom m{t-1}(1-\vartheta)^{bDt}
 \leq \frac{2}{m} \exp\left \{t(2\log m-bL_\vartheta)\right\}
 \leq \frac{2}{m} \exp\left \{t(2\log m - b A \log m)\right\},
$$
where the last inequality holds since $L_\vartheta \geq A \log m$. Set $\xi = bA - 2$, which is positive by~\eqref{eq:sublinear-matching-small-cut-constant}. It then follows that the first term in~\eqref{eq:sublinear-matching-hall-union-bound} is at most
$$
 \frac2m \sum_{t\geq1} m^{-\xi t}
 =\frac2m \cdot \frac{m^{-\xi}}{1-m^{-\xi}}
 =o(1).
$$
The second term in~\eqref{eq:sublinear-matching-hall-union-bound} is at most
$$
 2\exp\left(m\log4-\frac{\tau m}{4}L_\vartheta\right)
 \leq2\exp\left(m\log4-\frac{\tau A}{4}m\log m\right)
 =o(1).
$$
Note that both bounds depend only on $m$ and the constants $a$ and $A$, thus  establishing the asserted uniformity. Finally, $- \log(1-\vartheta) \geq \vartheta$ holds for $0<\vartheta<1$, thus verifying the sufficient condition stated at the end of the lemma.
\hfill$\square$


\subsection{Proof of Lemma~\ref{lem:percolated-expansion}}\label{app:tree-expansion-percolation}


We bound separately the probabilities that the two conditions
in Definition~\ref{def:expander} fail.  First, we verify the
convenient sufficient form of the discrepancy hypothesis.
Set
$$
 d = pn \qquad \textrm{and} \qquad
 \Lambda = -\log(1-\vartheta).
$$
Using the definitions of $\gamma_0$ and $\gamma_R$, and the assumed bounds $R \geq 2$ and $\beta \leq \gamma_R d$, a straightforward calculation shows that

\begin{equation}
\label{eq:expansion-discrepancy-bounds}
 \beta\leq d\min\left\{
  \frac{\alpha}{4(1+\sqrt R)},\,
  \frac18\sqrt{\frac{\alpha}{R+1}},\,
  \frac1{4R}
 \right\}.
\end{equation}

Starting with (E1), consider disjoint sets $X,Y\subseteq V(G)$ of sizes
$$
 1\leq x:=|X|<\frac{n}{2R} \qquad \textrm{and}
 \qquad |Y|=Rx,
$$
and set $Z = V(G)\setminus(X\cup Y)$.  We claim that
\begin{equation}
\label{eq:expansion-boundary}
\begingroup
\expandafter\let\expandafter\label\csname ltx@label\endcsname
\label{eq:sublinear-expansion-boundary}
\endgroup
 e_G(X,Z)\geq\kappa dx.
\end{equation}
Suppose first that $x\leq\alpha n/[4(R+1)]$.
The minimum degree and bijumbledness assumptions imply
\begin{align*}
 e_G(X,Z)
 &=\sum_{v\in X}\deg_G(v)-\ee_G(X,X)-e_G(X,Y) \geq \alpha dx-px^2-\beta x-pRx^2-\beta\sqrt R\,x\\
 &=\alpha dx-p(R+1)x^2-\beta(1+\sqrt R)x \geq\frac{\alpha}{2}dx
 \geq\kappa dx,
\end{align*}
where in the penultimate inequality we use the assumed upper bound on $x$ which implies $p(R+1)x^2\leq\alpha dx/4$, and the first estimate
in~\eqref{eq:expansion-discrepancy-bounds} which implies
$\beta(1+\sqrt R)x\leq\alpha dx/4$.

Suppose then that $x>\alpha n/[4(R+1)]$. Since
$x < n/(2R)$ and $R \geq 2$, we have
$$
 |Z|=n-(R+1)x
 >n\left(1-\frac{R+1}{2R}\right)
 \geq\frac n4.
$$
Consequently,
$$
 \sqrt{x|Z|}
 >\frac n4\sqrt{\frac{\alpha}{R+1}}.
$$
The second estimate
in~\eqref{eq:expansion-discrepancy-bounds} therefore ensures that
$\beta\leq p\sqrt{x|Z|}/2$ holds. The bijumbledness assumption, applied to 
$X$ and $Z$ implies
$$
 e_G(X,Z)
 \geq px|Z|-\beta\sqrt{x|Z|}
 \geq\frac12px|Z|
 \geq\frac18dx
 \geq\kappa dx.
$$
This establishes~\eqref{eq:expansion-boundary} in both ranges.

Suppose that the first expansion condition fails in
$G_\vartheta$. Let $X \subseteq V(G)$ be a set of size
$1 \leq x < n/(2R)$ and let $Y\subseteq V(G)\setminus X$
be a set of size $R x$ such that $\Gamma_{G_\vartheta}(X) \subseteq Y$.
Set $Z=V(G)\setminus(X\cup Y)$.


For each fixed pair $X,Y$, the independence of edge retention
shows that the probability that $E_{G_\vartheta}(X,Z) = \varnothing$ is
$$
 (1-\vartheta)^{e_G(X,Z)}
 =\exp\left(-\Lambda e_G(X,Z)\right)
 \leq\exp(-\Lambda\kappa dx),
$$
where the inequality uses~\eqref{eq:sublinear-expansion-boundary}.
For any given $x$, there are at most $n^x n^{Rx}$ choices
of $X$ and $Y$. It thus follows by a union bound and inequality~\eqref{eq:sublinear-expansion-retention-condition} that
\begin{align*}
 \PP[\text{the first expansion condition (E1) fails}]
 &\leq
 \sum_{1\leq x<n/(2R)}
 \exp\left((R+1)x\log n-\Lambda\kappa dx\right)\\
 &\leq\sum_{x\geq1}n^{-2x}
 =\frac1{n^2-1}.
\end{align*}



Suppose next that the second expansion condition (E2) fails. Set
$t = \lceil n/(2R)\rceil$ and let $A, B \subseteq V(G)$ be disjoint sets of size $t$ each such that $E_{G_\vartheta}(A,B) = \varnothing$. Since $t\geq n/(2R)$, the third
estimate in~\eqref{eq:expansion-discrepancy-bounds} ensures that 
$\beta\leq pt/2$ holds. It then follows by the bijumbledness of $G$ that
$$
 e_G(A,B)
 \geq pt^2-\beta t
 \geq\frac12pt^2
 \geq\frac{pn^2}{8R^2}.
$$
There are at most $\binom nt^2$ such pairs $A,B$ to consider.
The assumption $\log R = o(\log n)$ implies
$R=n^{o(1)}$ and $n/R\to\infty$; in particular,
$t \leq n/R$ holds for all sufficiently large $n$.
Additionally, $n/t \leq 2R$ holds by the definition of $t$.
Using the standard bound $\binom nt\leq(en/t)^t$, independence, and a 
union bound, we obtain
\begin{align*}
 \PP[\text{the second expansion condition (E2) fails}]
 &\leq
 \binom nt^2
 \exp\left(-\frac{\Lambda pn^2}{8R^2}\right)\\
 &\leq
 \exp\left(
   \frac{2n}{R}\log(2eR)
   -\frac{n(R+3)\log n}{8\kappa R^2}
 \right)\\
 &\leq
 \exp\left(-\frac{n\log n}{16\kappa R}\right)
 =o(1),
\end{align*}
where the second inequality holds since
$\Lambda pn \geq (R+3) \log n/\kappa$ by~\eqref{eq:sublinear-expansion-retention-condition}, and the third inequality holds since $\log R=o(\log n)$ implies that $2\log(2eR)\leq\frac{\log n}{16\kappa}$ holds for sufficiently large $n$.
We conclude that a.a.s. both expansion conditions hold. \hfill $\square$

\end{document}